\documentclass[11pt]{article}
\usepackage{amsmath,amssymb,amsthm,mathrsfs}
\usepackage{amsfonts}
\usepackage{graphicx}

\usepackage{frcursive}

\usepackage{appendix}
 \allowdisplaybreaks

 \usepackage{geometry}
\usepackage[colorlinks, linkcolor=black, anchorcolor=black,citecolor=black]{hyperref}

\usepackage{color}

\newcommand{\av}[1]{\langle #1\rangle}
\DeclareMathOperator*{\essinf}{ess\,inf}
\DeclareMathOperator*{\esssup}{ess\,sup}

\numberwithin{equation}{section}

\newcommand{\dd}{\mathrm{d}}

\newcommand{\R}{\mathbb{R}}

\newtheorem{theorem}{Theorem}[section]

\newtheorem{assumption}[theorem]{Assumption}

\newtheorem{definition}[theorem]{Definition}
\newtheorem{example}[theorem]{Example}

\newtheorem{lemma}[theorem]{Lemma}

\newtheorem{proposition}[theorem]{Proposition}
\newtheorem{remark}[theorem]{Remark}

\renewcommand{\eqref}[1]{(\ref{#1})}

\begin{document}

\title{\textsc{Spreading speeds for nonlocal Fisher-KPP equations
with time-dependent asymmetric kernels}}

\author{
Zhucheng Jin$^{a}$ and Tao Zhou$^{b,}$\thanks{ Corresponding author \\   
 E-mail addresses:   jinzc@ustc.edu.cn (Z. Jin);
 tzhou910@ustc.edu.cn (T. Zhou)}\\[0.4em]
{\footnotesize $^{a}$School of Mathematical Sciences, University of Science and Technology of China, Hefei, Anhui 230026, P.R. China}\\
{\footnotesize$^{b}$Center for Pure Mathematics, School of Mathematical Sciences,  Anhui University, Hefei, Anhui 230601, P.R. China} 
}

\maketitle

\begin{abstract}
In this paper, we establish spreading speeds for a nonlocal Fisher-KPP equation with  measurable time-dependent  asymmetric kernels, without imposing an inequality between   growth and   diffusion   or   requiring  the kernel to be  positive near the origin.
The kernel has a time independent upper bound which is thin-tailed. 
Both the kernel and   the  linear growth rate  have uniform means.  
Under  these assumptions, we determine  the  exact rightward and leftward spreading speeds and prove convergence to the positive equilibrium inside the spreading interval.
The proof is based on comparison  arguments and  a family of compactly supported subsolutions.
Explicit  examples  show  that   time-dependent dispersal kernels with the same time average  can  exhibit  different propagation behaviors. 
 

\vspace{0.2in}
\noindent \textbf{Key words}.  spreading speed; time-dependent kernel; asymmetric kernel; uniform mean.

\vspace{0.1in}
\noindent \textbf{2020 Mathematical Subject Classification}. 35K57; 35B40; 35B65; 92D25.

\end{abstract}

\section{Introduction}
We investigate  the large time behavior  of solutions to the  following equation 
\begin{equation} \label{Pb}
\begin{cases}
\displaystyle \partial_t u(t,x) = \int_{\R} K(t, x-y) [u(t,y) -u(t,x) ] \dd y  + u(t,x) f(t,u(t,x)), &  t>0, x\in\R, \\
\displaystyle  u(0,x) =u_0(x), & x\in \R,
\end{cases}
\end{equation}
where $u_0$ is a bounded and uniformly continuous function in $\R$ and  $0 \leq  u_0\leq 1$.   
The nonnegative kernel $K$    is  measurable in time and need not be symmetric or positive near the origin
The reaction term is  of Fisher-KPP type.

For the classical Fisher-KPP equation
$$u_t= du_{xx} +F(u),    $$
with $F(0)=F(1)=0$, Fisher~\cite{fisher1937wave} and Kolmogorov et  al.~\cite{kolmogorov1937study} studied  traveling wave solutions. Under the monostable assumption and the  inequality $F(u) \leq F'(0) u$, the minimal wave speed is $2\sqrt{dF'(0)}$.   Aronson and
Weinberger~\cite{aronson1978multidimensional} established the corresponding spreading speed result for the Cauchy problem.  
 For nontrivial  compactly supported initial data with $0 \leq u_0 \leq 1$, the solution satisfies  
\begin{equation*} 
\lim_{t\to +\infty } \sup_{ |x| \geq ct} u(t,x) =0,  \;\; \forall c>c_* \;\; \text{ and } \;\; \lim_{t\to +\infty } \inf_{ |x |\leq ct} u(t,x)=1,  \;\; \forall c\in (0,c_*).  
\end{equation*}
where $c_*:= 2\sqrt{dF'(0)}$ is the spreading speed. 
An   approach based on monotone semiflows  was  subsequently developed in  \cite{weinberger1982long, liang2007asymptotic}.

To consider dispersal over  long distance, the Laplacian  is replaced by  the  convolution operator
$$ \phi \mapsto   \int_{\R}K(x-y)[\phi(y) - \phi(x)] \dd y.$$
For the autonomous linearized equation,   the   exponential solution
 $e^{-\lambda(x-ct)}$ satisfies 
$$
\lambda c=\int_\R K(y)(e^{\lambda y}-1)\dd y+F'(0).
$$
This identity gives a   variational formula  for  the spreading speed.

Propagation phenomena in autonomous nonlocal  equations  have been  studied in \cite{schumacher1980travelling,carr2004uniqueness, coville2005propagation,lutscher2005effect,xu2021spatial,du2025asymptotic}.  For more information,  we refer to the recent monograph \cite{roquejoffre2024dynamics}.    
Unlike the local diffusion equations,   the nonlocal  diffusion operator introduces additional  difficulties in analysis because it generally lacks the regularizing  effect. 
Kernel asymmetry  also affects the signs of two spreading speeds.   Both ends of the spreading interval may move in the same direction. 
Xu et al. \cite{xu2021spatial}  studied this issue for autonomous KPP equations with a continuous kernel that  is positive near the origin. 
More recently,  Du et al. \cite{du2025asymptotic} developed a new approach relying on spectral properties   to establish  spreading speeds for autonomous equations.    The key point  in \cite{du2025asymptotic} is the limit of the principal eigenvalue of a convolution operator with a drift term   on expanding intervals.

Temporal heterogeneity models seasonal changes and other environmental variations. 
 For local diffusion,  generalized transition front and spreading speeds were studied in \cite{nadin2012propagation, shen2010variational,shen2011existence}, see  \cite{berestycki2008asymptotic, berestycki2019asymptotic} for more general space-time heterogeneities.  
 For  nonlocal equations, we refer to 
\cite{jin2009spatial, jin2012seasonal, liang2006spreading} in   time periodic environment,  to  \cite{shen2016transition, ducrot2021generalized, dj2023spreading, bao2025epidemic, wang2026generalized} in  nonperiodic case,   and to \cite{shen2017stability, tao2026} for more general space-time heterogeneities.     In  \cite{dj2023spreading}, the author investigated  autonomous kernel, their proof  combines uniform spatial continuity estimates with a persistence lemma and assumes that the average of growth rate large than the integral of dispersal kernel.   Under the large diffusion rate,  \cite{bao2025epidemic} obtained a scalar spreading speed for a symmetric and compactly supported  nonautonomous  kernel as part of their analysis of an epidemic system.   Our spreading speed results follow from a direct subsolution construction and do not require an inequality between the linear growth rate and the diffusion rate. 
 
\vspace{1ex}

We now   describe our   main results.  
 Set  $r(t)=f(t,0)$ and  let $\overline K$ and $\av r$ denote the uniform mean  of $K(t,y)$ and $r(t)$, respectively. 
Assume   $\int_{\R} \overline{K}(y)\dd y>0$.  
 We require that 
 $k(y):= \esssup_{ t\geq 0} K(t,y)$ is exponentially bounded.  The  exponential   supersolutions  for upper estimate have decay rate in the interior of 
$$
 \mathcal I=\left\{\lambda:\int_\R
k(y)e^{\lambda y}\dd y<\infty\right\}.
$$
 For these parameters, temporal averaging gives the function
 $$
 H(\lambda)=\int_\R\overline K(y)(e^{\lambda y}-1)\dd y+\av r.
 $$
 The candidate rightward spreading  speed is 
 $$\inf_{\lambda>0,\,\lambda\in\mathcal I}
\frac{ H(\lambda)}{\lambda},$$
  reflection gives the candidate leftward speed.
  Proposition~\ref{prop:classification}  distinguishes four cases in each direction:   an interior minimum, an infimum approached at infinity, and two finite endpoint cases distinguished by whether $H'(b) = H(b)/b$ where $b$ is the finite endpoint  of $\mathcal{I}$.   We  use Example~\ref{ex:periodic-direction-switching}   to explain these cases. 
  In Theorem~\ref{thm:spreading}, we establish the exact  rightward and leftward
   spreading speeds in all four cases. 
   An additional  assumption  is
   required only in a direction belonging to Case~$(iv)$ for lower estimate, and the two
   directions may belong to different cases.

    The upper estimate is obtained by the  exponential decay solution.

For the lower estimate, we construct a family of compactly supported subsolutions whose centers vary continuously with the exponential parameter. These subsolutions give a positive lower bound and their  centers cover the required spreading interval. 
This construction avoids the persistence lemma and the regularity estimates used in \cite{bao2025epidemic,dj2023spreading}.  We then adapt an argument in \cite{liang2020jfa} to obtain convergence to $1$ on spreading  intervals.


 Time dependence of the dispersal kernel can affect propagation phenomena. Example~\ref{ex00}  shows that   sublinear  spreading  can occur when  $\overline{K}=0$ almost everywhere on $\R$. 
   Example~\ref{ex1} gives two 
kernels with the same time average and the same essential  supremum in time,  while the two solutions  have different limits along the path $x=t/2$.   The  difference comes from the durations of transient long distance dispersal.  For  autonomous kernels,  and  for separable kernels $K(t,y)=d(t)J(y)$, the  phenomenon in Example~\ref{ex1} does not happen.

%
%

The assumptions and main  results  are  stated in Section~\ref{sec2}.  Section \ref{sec:pre} gives the preliminary lemmas. 
  In section~\ref{sec:proof}, we prove  spreading speed  theorems. Section~\ref{sec:ex}  presents examples illustrating the endpoint cases and the effect of time-dependent kernel.

\section{Assumptions and main results} \label{sec2}

We first introduce the concept of uniform mean which plays an important role in overcoming the difficulties induced by time heterogeneities, see \cite{berestycki2019asymptotic, nadin2012propagation,shen2011existence}.   
Most time heterogeneities admit a mean value, for instance periodic, almost periodic and  uniquely ergodic media.

\begin{definition}
A bounded measurable function $h$ has a uniform mean $\av h$ if
\begin{equation}\label{eq:mean}
\lim_{T\to\infty}\sup_{s\ge0}
\left|\frac1T\int_s^{s+T}h(t)\dd t-\av h\right|=0.
\end{equation}
\end{definition}

\vspace{1em}

 We impose the following  assumptions  on  $f$.

\begin{assumption}\label{ass:f}
The function $f:[0,\infty)\times[0,1]\to\R$ is measurable in $t$ for each
$u$, and is Lipschitz continuous  in $u$ with a constant $L>0$ independent of $t$.
Assume that
\begin{equation}\label{eq:kpp}
\begin{split}
 f(t,1)=0,  \qquad &   r(t):= f(t,0) \text{ has  a uniform mean $\av{r}$,} \\ 
  \essinf_{t\ge0}f(t,u)>0,   \qquad &
f(t,u)\leq  r(t)   \quad \forall 0 <  u < 1.
\end{split}
\end{equation}
\end{assumption}

 Assumption~\ref{ass:f} implies   
\begin{equation}\label{eq:f-bounds}
\begin{split}
& |f(t,u)| \leq L (1-u),   \qquad   r(t)-Lu \le f(t,u)\le r(t),\qquad  \\
&  r_0:=\essinf_{t\ge 0} r(t)\ge\essinf_{t\geq 0} f(t,1/2)>0.
\end{split}
\end{equation}
In particular,   $f$ and $r$  are   bounded.

\vspace{1em}

We next state the assumptions  on the  dispersal kernel. 

\begin{assumption}\label{ass:k}
The nonnegative kernel $K=K(t,y)$ is Lebesgue measurable. Set
$$
k(y):=\|K(\cdot,y)\|_{L^\infty(0,\infty)}.
$$
For some $\eta>0$,
\begin{equation}\label{eq:envelope}
\int_\R k(y)e^{\eta|y|}\dd y<\infty.
\end{equation}
For almost every $y$, $K(\cdot,y)$ has a uniform mean
$\overline K(y):=\av{K(\cdot,y)}$ and $\int_{\R} \overline{K}(y) \dd y >0$.
\end{assumption}

We use the following notations in the main results. 
Define
\begin{equation}\label{eq:mass}
k_0(t):=\int_\R K(t,y)\dd y.
\end{equation}
By \eqref{eq:envelope}, the dominated converge theorem  and Fubini-Tonelli theorem imply 
$$ \av{k_0}=\int_\R\overline K(y)\dd y.$$
Denote 
\begin{equation} \label{eq:domain}
\mathcal I :=\left\{\lambda\in\R:
\int_\R k(y)e^{\lambda y}\dd y<\infty\right\},\qquad
i_-:=\inf\mathcal I,\quad i_+:=\sup\mathcal I. 
\end{equation}
By \eqref{eq:envelope},  $\mathcal{I}$ is an interval containing  $[-\eta, \eta]$. 

We next define the linearly  determined speed formula.
Let  
\begin{equation}\label{eq:H} 
H(\lambda):=\int_\R\overline K(y)e^{\lambda y}\dd y- \av{k_0} +\av r,
\qquad \lambda \in ( i_-, i_+),  
\end{equation}
and 
\begin{equation}\label{eq:speeds}
c_+^* :=\inf_{0<\lambda<i_+}\frac{H(\lambda)}{\lambda},\qquad
c_-^*:=\inf_{i_-<\lambda<0}\frac{H(\lambda)}{-\lambda}.
\end{equation}

 \vspace{1ex}
 The following proposition  completely characterizes all possible  locations of infimum of $H(\lambda)/\lambda$ on $(0, i_+)$.
\begin{proposition}\label{prop:classification}
 Set 
$$H(i_+) :=  \lim_{\lambda\uparrow i_+ }H(\lambda) \;\; \text{ and  }\;\;    H'(i_+):=\lim_{\lambda\uparrow i_+} H'(\lambda).  $$
In the positive direction, exactly  one   of the following cases  occurs.
\begin{itemize}
\item[(i)] The infimum of $H(\lambda)/\lambda$ is attained at a unique $p  \in(0,i_+)$, and
$pH'(p)=H(p)$.
\item[(ii)] $i_+=+\infty$, $\overline K=0$ almost everywhere on $(0,+\infty)$,   the quotient $H(\lambda)/\lambda$ is strictly decreasing on $(0, +\infty)$,
  and
\begin{equation}\label{eq:infinite-boundary}
c_+^*=0,\qquad H'(+\infty)=0.
\end{equation}
\item[(iii)]  $i_+<+\infty$, the quotient $H(\lambda)/\lambda$ is strictly
decreasing on $(0,i_+)$,  and 
\begin{equation}\label{eq:finite-boundary}
H(i_+)<+\infty,  \quad H'(i_+)< +\infty,\quad   c_+^*=\frac{H(i_+)}{i_+}=H'(i_+).
\end{equation}
\item[(iv)] $i_+<+\infty$, the quotient $H(\lambda)/\lambda$  is strictly decreasing
on $(0,i_+)$,  
and
\begin{equation}\label{eq:strict-endpoint}
H(i_+)<+\infty,  \quad H'(i_+)<+\infty,\quad   c_+^*=\frac{H(i_+)}{i_+}> H'(i_+).
\end{equation}
\end{itemize}
The classification in the  negative direction is obtained through replacing $H(\lambda)$ by $H(-\lambda)$.
\end{proposition}

 \begin{remark}
 Case $(i)$ is an interior minimum.  In   Case $(ii)$-$(iv)$, the infimum is approached at an endpoint of   $\mathcal{I}$.
 If $i_+<\infty$ and either $H(i_+)=\infty$ or $H'(i_+)=\infty$, the minimum is attained in the interior  as the proof below shows.
 
 \end{remark}

The following theorem gives the spreading  speed result in all four cases of   Proposition~\ref{prop:classification}.  
An additional assumption is  imposed only in a direction belonging to Case~$(iv)$.  

\begin{assumption}\label{ass:endpoint}
If Case~(iv) of Proposition~\ref{prop:classification} occurs in the
positive direction, assume
\begin{equation}\label{eq:tail-condition}
\int_0^\infty \overline K(y)e^{\mu y}\dd y=\infty
\qquad\text{for every }\mu>i_+.
\end{equation}
If Case~$(iv)$ of Proposition \ref{prop:classification}   occurs in the negative direction, assume
\begin{equation}\label{eq:left-tail-condition}
\int_{-\infty}^0 \overline K(y)e^{\mu y}\dd y=\infty
\qquad\text{for every }\mu<i_-.
\end{equation}
This condition  is  not required  in a direction belonging to
Cases~$(i)$-$(iii)$.
\end{assumption}

\begin{theorem}\label{thm:spreading}
Let Assumptions~\ref{ass:f}, \ref{ass:k}, and~\ref{ass:endpoint} hold.
Let $c_+^*$ and $c_-^*$ be defined by~\eqref{eq:speeds}.
Then $c_+^*+c_-^*>0$.
For every uniformly  continuous initial function $u_0 \not\equiv 0$ and $0\leq u_0\leq 1$,    the solution of \eqref{Pb}  satisfies:  
\begin{itemize}
\item[(i)] If $u_0$ is compactly supported, then, for every
$\varepsilon>0$,
\begin{equation}\label{eq:upper}
\lim_{t\to\infty}\sup_{x\ge(c_+^*+\varepsilon)t}u(t,x)=0 \quad \text{ and } \quad \lim_{t\to\infty}\sup_{x\le(-c_-^*-\varepsilon)t}u(t,x)=0.
\end{equation}
\item[(ii)] For every $-c_-^*<a<b<c_+^*$,
\begin{equation}\label{eq:lower-one}
\lim_{t\to\infty}\sup_{at\le x\le bt}|u(t,x)-1|=0.
\end{equation}
\end{itemize}
The two directions may belong to different cases of
Proposition~\ref{prop:classification}.
\end{theorem}

\begin{remark}
We emphasize that $\int_{\R}\overline{K} (y) \dd y >0$ is in order to ensure $c^*_+ + c^*_- >0$.   If $\overline{K}=0$ almost everywhere on $\R$, Example~\ref{ex00} gives that $c^*_+=c^*_-= 0$ and the sublinear propagation can happen.   The solution can converge to $1$ locally  when linear speeds vanish, $c^*_{\pm}=0$.
\end{remark}

\begin{remark}
Assumption~\ref{ass:endpoint}  is required only in a direction belonging to Case~$(iv)$. The theorem includes interior
minima and infima approached at infinite or finite endpoints.
Define
$$
\mathcal I_{\rm av}:=\left\{\lambda\in\R:
\int_\R\overline K(y)e^{\lambda y}\dd y<\infty\right\}.
$$
Since $\overline K\le k$, condition~\eqref{eq:tail-condition} is
equivalent to $\sup\mathcal I_{\rm av}=\sup\mathcal{I}=i_+$, and
condition~\eqref{eq:left-tail-condition} is equivalent to
$\inf\mathcal I_{\rm av}=\inf \mathcal{I}=i_-$, in the respective directions of
Case~(iv). These conditions hold directly  for  separable kernels
$K(t,y)=d(t)J(y)$ with $\av d>0$.  The infima cannot be approached at infinity in both directions, since $ \int_{\R}  \overline{K}(y) \dd y>0$.

The upper estimates~\eqref{eq:upper} require only
Assumptions~\ref{ass:f}  and~\ref{ass:k}.  The additional  Assumption~\ref{ass:endpoint}   is used to obtain the lower estimates in Case~$(iv)$.
\end{remark}

\section{Preliminaries}\label{sec:pre}
In  this section, we establish    well-posedness,   comparison principles and positivity  for \eqref{Pb},   and prove the properties of variational  speed formula  and  uniform mean.

\subsection{Well-posedness and comparison principles}
\begin{proposition}[Well-posedness]\label{prop:wellposedness}
Let Assumptions~\ref{ass:f} and \ref{ass:k} hold, and set
$$
X=BUC(\mathbb R),\qquad
X_1=\{v\in X:0\le v\le1\}.
$$
Herein $BUC(\R)$ denotes the Banach space  of  bounded and uniformly continuous functions on $\R$, equipped with the supremum norm. 
For every $\tau\ge0$ and $u_\tau\in X_1$, 	the  equation
 \begin{equation*}
 \begin{cases}
   \displaystyle \partial_t u(t,x) = \int_{\R} K(t, x-y) [u(t,y) -u(t,x) ] \dd y  + u(t,x) f(t,u(t,x)), &  t>\tau, x\in\R, \\
   \displaystyle  u(\tau,x) =u_\tau (x), & x\in \R,
 \end{cases}
 \end{equation*}
has  a unique global  solution and 
$$
u\in C([\tau,\infty);X)\cap W^{1, \infty}_{\rm loc} ([\tau, \infty); X),  \qquad  u(t,\cdot) \in X_1.
$$
\end{proposition}

\begin{proof}
The proof is standard. We sketch the proof. 
Set 
$$
\mathcal A(t)v :=K(t,\cdot)*v-k_0(t)v,
\qquad
M_K:=\int_{\mathbb R}k(y)\dd y.
$$
Then $\mathcal A(t)$ is a bounded linear operator on $X$
with $\|\mathcal A(t)\|\le2M_K$ for almost every $t$.

Let
$$
p(s):=\min\{1,\max\{0,s\}\},
\qquad
\widetilde F(t,s):=p(s)f(t,p(s)).
$$
Since $f(t,1)=0$ and $f(t,u)$ is  Lipschitz continuous 
in $u\in[0,1]$ uniformly in $t$, $\widetilde F(t,\cdot)$ is
globally  Lipschitz continuous uniformly in $t$. Hence the map
$$
G(t,v)=\mathcal A(t)v+\widetilde F(t,v(\cdot))
$$
is globally  Lipschitz continuous  in $v\in X$, uniformly in $t$, and measurable in $t$ for each $v$.

Using the integral formulation, we have 
 $$ u(t)=u_\tau+\int_\tau^tG(s,u(s))\dd s  $$
 By iteration argument and contraction mapping theorem, the integral equation has a 
  unique solution on a sufficiently  small interval.  
Then one can extend the solution on  $[\tau, \infty)$.   In addition, $u=0$ and $u=1$ solve the extended equation. 
Applying Lemma \ref{lem:comparison} whose proof uses only
the differential inequality and does not depend on this proposition, gives $0\le u\leq 1$.
The integral equation and
the uniform bounds also give the stated time regularity.
%
%
\end{proof}

We use  comparison principle on the whole line and on a moving  interval. 
\begin{lemma}\label{lem:comparison}
Fix $T>0$.  Let  $a\in L^\infty((0,T)\times\R)$, and let either
$\mathcal{D}(t)=\R$ or $\mathcal{D}(t)=(X_1(t),X_2(t))$, where $X_1,X_2\in C([0,T])$
and $X_1<X_2$.   Suppose $w$ is bounded, and  continuous on $[0,T]\times \R$, and  locally absolutely continuous in
time for each fixed $x$. 
Assume that,  for almost every $t\in [0, T]$ and each fixed  $x\in \mathcal{D}(t)$,
\begin{equation}\label{eq:comparison}
w_t(t,x)\ge\int_\R K(t,x-y)[w(t, y)-w(t,x)] \dd y+a(t,x)w(t,x). 
\end{equation}
If $w(0,x)\ge0$ on $\mathcal{D}(0)$ and, in the interval case,
$w(t,x)\ge0$ for $x\notin \mathcal{D}(t)$ and $t\in [0, T]$, then $w(t,x)\ge0$ for $t\in[0,T]$ and $x\in\mathcal{D}(t)$.
\end{lemma}

\begin{proof}
Set
$$
 K_*:=\|k_0\|_{L^\infty(0,T)}, \qquad
 \gamma:= \|a \|_{L^\infty((0,T) \times \R) } +  K_*+1,
$$
where $k_0(t)=\int_\R K(t,y)\dd y$. Define
\[
 v(t,x):=e^{-\gamma t}w(t,x),\qquad
 b(t,x):=k_0(t)+\gamma-a(t,x).
\]
Then, for almost every $t\in [0, T]$ and $x\in\mathcal{D}(t)$,
$$
 v_t(t,x)+b(t,x)v(t,x)
 \ge \int_\R K(t,y)v(t,x-y)\dd y,
 \qquad b(t,x)\ge K_*+1.
$$
 
Suppose  that $v$ takes a negative value. Since $w $ is
bounded and   continuous, the number
$$
 m:=\inf_{(t,x)\in[0,T]\times\R}v(t,x)
$$
is finite and $m <0$.  
 It is not necessary that this infimum is   attained.   The nonnegativity of $K$ gives
$$
 \int_\R  K(t,y)v(t,x-y)\dd y   \ge m k_0(t) \ge mK_*.
$$
Hence, 
$$
 v_t(t,x)+b(t,x)v(t,x)\ge mK_*,    \quad t\in [0, T], x\in \mathcal{D}(t). 
 $$

Fix $t\in(0,T]$ and $x\in  \mathcal{D}(t)$.  In the whole  line case, set $\tau=0$.  In the interval case, we define 
$$
\tau:=\sup \{s\in[0,t]:x\notin  \mathcal{D}(s)\}.
$$
If  $\{s\in[0,t]:x\notin  \mathcal{D}(s)\}= \emptyset$, let $\tau =0$.
The continuity of $X_1$ and $X_2$ yields  that $\tau<t$ and 
$$
x\in \mathcal{D}(s)\qquad\text{for all }s\in(\tau,t].
$$
If $\tau>0$, then $x\in\partial  \mathcal{D}(\tau)$, so the exterior
condition gives $v(\tau,x)\ge0$.
If $\tau=0$, the same inequality follows from the initial 	or exterior  condition.
Integrating  the differential inequality
from $\tau$ to $t$  yields
 \begin{align*}
 v(t,x)
 &\ge e^{-\int_\tau^t b(r,x)\dd r}v(\tau,x)
   +mK_*\int_\tau^t
       e^{-\int_s^t b(r,x)\dd r}\dd s\\
 &\ge mK_*\int_\tau^t
       e^{-\int_s^t b(r,x)\dd r}\dd s.
 \end{align*}
If needed, one can first integrate from $\tau+\varepsilon$ and let 
$\varepsilon\downarrow0$, using continuity at $\tau$.
Since $b\ge K_*+1$,
$$
 0\le\int_\tau^t e^{-\int_s^t b(r,x)\dd r}\dd s
 \le\int_\tau^t e^{-(K_*+1)(t-s)}\dd s
 \le\frac1{K_*+1}.
$$
As $mK_*\le0$, it follows that
$$
 v(t,x)\ge \frac{K_*}{K_*+1}\,m.
$$
This estimate also holds at $t=0$ and, in the moving  interval case, for
$x\notin \mathcal{D}(t)$, because $v$ is nonnegative there. Thus,   taking the infimum gives
$$
 m\ge\frac{K_*}{K_*+1}\,m> m.
$$
This contradiction proves $w\ge0$.

\end{proof}

To compare with the compactly supported subsolutions
constructed below, we need the solution to be positive on a sufficiently
long interval. The following lemma provides such an interval at a positive
time for every nontrivial initial datum. 

\begin{lemma}\label{lem:positive}
Let Assumptions ~\ref{ass:f} and \ref{ass:k} hold and $u_0\in X_1\setminus \{ 0\}$. 
 Let $u$ be a solution of \eqref{Pb}.   Then,  for every $l>0$, there exists  $T_0>0$ and $x_l  \in \R$   such that 
$$
u(T_0,x)>0   \qquad    \text{ for all }  x \in [x_l-l,x_l+l].
 $$
\end{lemma}

\begin{proof}
Fix $l>0$.   Choose $T_0,M>0$ and a bounded interval $[a,b]$  such that
$$
 j(t,y):= \min\{K(t,y),M\}\mathbf1_{[0,T_0]}(t) \mathbf1_{[a,b]}(y)
 $$
 is nonzero in $L^1((0, T_0) \times \R)$.  The interval  $[a,b]$ does  not require  to include $0$.  Let   $q(y)=\int_0^{T_0}j(t,y)\dd t$.  Note  that  $q\in L^1\cap L^\infty$, $q\ge0$, and $q\not\equiv0$.  Since 
  $$
  \|(q*q)(\cdot+h)-q*q\|_\infty
  \le\|q\|_\infty\|q(\cdot+h)-q\|_1\longrightarrow0,  \text{ as } h \to 0,
  $$
 we have that    $q*q$ is continuous.     In addition,  we notice 
 $$  \int_{\R} q*q(x) \dd x = \left(\int_{\R} q(x) \dd x \right)^2 >0. $$
Hence, we can find  a positive constant $\nu>0$ and an interval $[\alpha,\beta]$ with $\alpha<\beta$ such that 
$$  q*q(x) > \nu   \qquad \text{ for all } x\in [\alpha,\beta]. $$

Let  $C\ge\|k_0\|_\infty+\|f\|_\infty$. Set  $U(t,x)=e^{Ct}u(t,x)$.  We have 
\begin{align*}
U_t & = \int_{\R}K(t,x-y) U(t,y) \dd y +(C -k_0(t) + f(t,e^{-Ct} U)) U  \\
&  \ge \int_{\R}  j(t,x-y) U(t,y) \dd y.
\end{align*} 
  For brevity, denote  $j_s:=j(s,\cdot)$. Then
  $$
U(t,x)\ge  u_0+\int_0^t  j_{s_1}*U(s_1, x)\dd s_1,
\qquad  \forall  0\le t\le T_0.
$$
In particular, $U(t, \cdot)\ge u_0$.
 By iteration argument, we have 
 $$
 U(t,x)\ge
 \int_{\Delta_n(t)}
 (j_{s_1}*\cdots*j_{s_n}*u_0)(x)\, \dd s_1\cdots \dd s_n.
$$
 where 
$$
\Delta_n(t):=
\{(s_1,\ldots,s_n)\in[0,t]^n:
 0<s_n<\cdots<s_1 <t   \}.
$$
By the symmetry of integral in $s_1,\ldots,s_n$ and the commutativity of spatial  convolution, we have 
 $$
U(T_0,x)\ge  \frac{1}{n!} \int_{[0,T_0]^n} (j_{s_1}*\cdots*j_{s_n}*u_0)(x) \, \dd s_1\cdots \dd s_n.
$$
Taking $n =2m$,  we obtain 
 \begin{equation}\label{eq:series}
 U(T_0,\cdot)\ge\frac1{(2m)!}q^{*(2m)}*u_0,
 \end{equation}
 where 
 $$
 q^{*(2m)}*u_0(x):=  \int_{\R^{2m}} \left(  \prod_{i=1}^{2m} \int_0^{T_0}j(t,y_i)\dd t \right) \; u_0\left(x-\sum_{i=1}^{2m} y_i \right) \dd y_1\dots \dd y_{2m}.
 $$

Since $u\not \equiv 0$ is continuous and nonnegative,    there are $A<B$ and $\delta>0$  such that  $u_0\ge\delta$ on $[A,B]$.
 The right-hand side of \eqref{eq:series} is positive on
$$
 (A+m\alpha,B+m\beta).
 $$
 The length of this  interval tends to infinity as $m \to \infty$.   
  Choose $m$   large enough     such  that 
  $$
  B+m\beta - A- m\alpha >2l. 
  $$
  Let 
  $$x_l = \frac{A +B+ m (\alpha + \beta)}{2}. $$
We have 
  $$
   [x_l -l,x_l+l] \subset  (A+m\alpha,B+m\beta),
  $$
  and $U(T_0, x)>0$  for $x\in [x_l -l, x_l+l]$.
  The proof is complete. 
\end{proof}

 \subsection{Properties of  the  speed formula and uniform means}
 We first prove Proposition \ref{prop:classification}. 
 
 \begin{proof}[Proof of Proposition \ref{prop:classification}]
 Denote   $g(\lambda):= \lambda H'(\lambda) - H(\lambda)$ and 
 $$c(\lambda) := \frac{H(\lambda)}{\lambda}, \quad   \lambda \in (0, i_+).$$
Differentiation gives 
 \begin{equation*}
 \frac{d c(\lambda)}{d \lambda} = \frac{\lambda H'(\lambda)  - H(\lambda)   }{\lambda^2}= \frac{g(\lambda)}{\lambda^2}, 
 \end{equation*}
 and 
 $$g'(\lambda) = \lambda H''(\lambda) = \lambda  \int_{\R}\overline{K}(y) e^{\lambda y} y^2  \dd y.$$
 The condition $\overline{K}\not \equiv 0$ in $L^1(\R)$  yields $g'(\lambda) >0$ on $(0, i_+)$.    
 Since $H(0)= \av{r}>0$, 
 we have $g(0)=- \langle r \rangle <0$ and    $c(\lambda)\to+ \infty$ as $\lambda \downarrow 0$.     If $g$ has a zero   $p\in (0, i_+)$, then the zero $p$ is unique by the continuity and monotonicity of $g$.  Hence,  $H(\lambda)/\lambda$ has a unique minimizer  $p$ and $pH'(p)=H(p)$.   
 If there is no zero, then the continuity and monotonicity of $g$ imply that $H(\lambda)/\lambda$  is strictly decreasing on $(0, i_+)$.   
 
Now we suppose that $g$ has no zero. Hence  the infimum is approached at $i_+$.    If $i_+ = +\infty$ and $\int_{p}^{q}\overline{K}(y)\dd y>0$ with $p<q$,   then  
 $$ H(\lambda)  \geq e^{p \lambda }    \int_{p}^{q}   \overline{K}(y) \dd y - \av{k_0} + \av{r},   $$
 and 
 $$ \frac{H(\lambda) }{\lambda} \to +\infty \text{ as  } \lambda \to +\infty. $$
 This implies that the infimum is attained at an interior point.  So $\overline{K} =0$ almost everywhere on  $(0, +\infty)$ when $i_+= \infty$.  
It follows  that 
 $$ c^*_+ = \lim_{\lambda\to +\infty} c(\lambda) =  \lim_{\lambda\to +\infty}   \frac{\int_{-\infty}^{0}  \overline{K}(y) e^{\lambda y} \dd y - \av{k_0} +\av{r}}{\lambda}=0,  $$
 and 
 $$   \lim_{\lambda\to +\infty}  H'(\lambda) = \lim_{\lambda\to +\infty} \int_{-\infty}^{0}  \overline{K}(y) e^{\lambda y} y \dd y  =0. $$

 Assume $i_+<+\infty$.   Since $\int_{\R}\overline{K}(y)e^{\lambda y} \dd y>0$, we obtain $H(\lambda) > \av{r} - \av{k_0}$  on $(0, i_+)$.   Hence, 
 $$c(\lambda) > \frac{ \av{r} - \av{k_0} }{ \lambda }.$$
 We observe  that $c(\lambda)$ is decreasing and bounded below  on $[i_+/2, \;  i_+)$. So $H(i_+)<\infty$. 
 From $H''(\lambda) >0$, we  have $H'(\lambda) > H'(i_+/2)$ for $\lambda 
 \in (i_+/2, i_+)$.    It follows from $g(\lambda)<0$ for $\lambda\in (0, i_+)$  that   $H'(\lambda) < H(\lambda) / \lambda$. So $H'(\lambda)$ is bounded on  $(i_+/2, i_+)$.   We obtain  $H'(i_+)<+\infty$.
 Letting  $\lambda \to i_+$ in $g(\lambda)<0$ yields 
 $$  i_+H'(i_+) - H(i_+) \leq 0.  $$ 
 The equality gives Case $(iii)$, while the strict inequality gives Case $(iv)$.   The proof is complete.
 
 \end{proof}

We next show  the property of uniform mean.  The following result is  related to the characterization of least means in  \cite[Lemma 3.2]{nadin2012propagation}.  
 
 \begin{lemma}\label{lem:means}
 Let the interval  $ [p,q]\subset  \R $ and let $g_\lambda \in L^\infty (0, \infty)$ be uniformly bounded  for $\lambda \in[p,q]$.   Assume that there exist constants
  $\langle g_\lambda\rangle $ 
  such that
  $$
  \lim_{T\to\infty}
  \sup_{\lambda\in [p,q]}\sup_{t\ge0}
  \left|
  \frac1T\int_t^{t+T}g_\lambda(s)\dd s
  -\langle g_\lambda\rangle
  \right|=0,
  $$
  and that
  $$
  \sup_{\lambda\in [p,q]}\langle g_\lambda\rangle<0.
  $$
   Then there exist $\theta>0$ and  a family   $a_\lambda\in W^{1,\infty}(0,\infty)$,  
 such that
 $$\sup_{ \lambda\in [p,q]}  \|a_\lambda  \|_{W^{1, \infty} (0, \infty)} <\infty$$
 and 
 \begin{equation}\label{eq:corrector}
 a_\lambda'(t)+g_\lambda(t)\leq - \theta,  \quad  \forall \lambda \in [p,q], \;\; a.e.  \;  t\geq 0.
 \end{equation}
 \end{lemma}
 
 \begin{proof}
  Choose $\theta>0$ such that
  $$
  \sup_{\lambda\in [p,q] }  \langle g_\lambda \rangle
  \leq -2\theta.
  $$
  From  assumption of $g_\lambda$, there exists sufficiently large $T>0$, independent of $\lambda$, such that
  $$
  \frac{1}{T} \int_t^{t+T} g_\lambda(s)\dd s  \leq -\theta
  $$
  for all $t\ge0$ and $\lambda\in [p,q]$.
  
Define
  $$
  a_\lambda(t)
  :=\frac1T \int_0^T\int_0^t
  \bigl[g_\lambda(\tau+s)-g_\lambda(\tau)\bigr]
  \dd \tau\dd s.
  $$
  Differentiating with respect to $t$, we obtain
  $$
  \begin{aligned}
  a_\lambda'(t)+g_\lambda(t)
  &=\frac1T\int_0^T g_\lambda(t+s)\dd s\\
  &=\frac1T\int_t^{t+T}g_\lambda(s)\dd s
  \le -\theta
  \end{aligned}
  $$
  for almost every $t\ge0$.
  
  It remains to verify the uniform boundedness.
  Changing the order of integration gives
  $$
  a_\lambda(t)
  =\frac1T\int_0^T
  \left[
  \int_t^{t+s}g_\lambda(\tau)\dd \tau
  -\int_0^s g_\lambda(\tau)\dd \tau
  \right]\dd s.
  $$
Set 
  $$
  M:=\sup_{\lambda\in [p,q]}
  \|g_\lambda\|_{L^\infty(0,\infty)},
  $$
then  
  $$
  |a_\lambda(t)|
  \le\frac1T\int_0^T 2Ms\dd s=MT,
  $$
  and
  $$  
  |a_\lambda'(t)|\le 2M. 
  $$  
Hence,  $a_\lambda$ has all the desired properties.   We complete the proof.
  
 \end{proof}

  \section{Proof of  Theorem~\ref{thm:spreading}}\label{sec:proof}
  
  \subsection{Upper estimates}\label{sect:up}
  Let $u_0$ be compactly supported, with $0\le u_0\le1$.
  We first prove the rightward estimate. Fix $\varepsilon>0$.
  By the definition of $c_+^*$, choose   $\lambda\in(0,i_+)$ such that
  $$
  \frac{H(\lambda)}{\lambda}<c_+^*+\frac{\varepsilon}{2}.
  $$
  Set
  $$
  m_\lambda(t)  :=\int_\R K(t,y)(e^{\lambda y}-1)\dd y+r(t).
  $$
  We can verify that $ \overline{v}(t,x) := A\exp\left\{-\lambda x+\int_0^tm_\lambda(s)\dd s  \right\}$  solves the following  linearized equation, 
  \begin{equation*}
  \begin{split}
  \partial_t \overline{v} &- \int_{\R} K(t, x-y) [ \overline{v}(t,y)-  \overline{v}(t,x)] \dd y - r(t)  \overline{v} \\
  &= \overline{v}  \left\{  m_\lambda (t)  - \int_\R K(t,y)[e^{\lambda y}-1] \dd y - r(t)   \right\}  \\
  & = 0. 
  \end{split}
  \end{equation*}
 Define
   $
   \overline{u}:=\min \left\{   1,   \overline{v}   \right\}.   
   $
   We know that $\overline{u}$ is a supersolution of \eqref{Pb}.
  Choose $A$ large enough  such  that
  $\overline{u}(0,\cdot)\ge u_0$. 
  Since $f(t,u)\le r(t)$,    the comparison principle yields  $u\le \overline{u}$.  
  The function  $m_\lambda$ has a uniform mean  $\av{m_\lambda}=H(\lambda)$.  This implies 
  \begin{equation*}
  \int_{0}^{t} m_\lambda(s) \dd s = H(\lambda) t  + o(t)  \quad \text{ as }  t \to \infty.
  \end{equation*}
  Therefore, for $x\ge(c_+^*+\varepsilon)t$, 
  \begin{equation*}
   u(t,x) \leq \overline{u} (t,x) \leq A\exp\left\{ -\lambda (c_+^*+\varepsilon)t + H(\lambda) t+ o(t)    \right\}  \leq A\exp\left\{ -\frac{ \lambda  \varepsilon}{2} t + o(t)    \right\}   \to 0,
  \end{equation*}
   as $t\to \infty.$ 
   This proves 
   \begin{equation*}
   \lim_{t\to\infty}\sup_{x\ge(c_+^*+\varepsilon)t}u(t,x)=0,
   \end{equation*}
     By the symmetric argument, we can prove the other limit in \eqref{eq:upper}.  
  
  \subsection{A family of compactly supported subsolutions}
  For the lower estimate, we  construct a family compactly supported  positive subsolutions.  
  Since the same  construction will be used for truncation in   Section~\ref{sect:truncate},  we consider  a smaller dispersal kernel while keeping  $-k_0(t) \phi$  in the following lemma. 
   
  \begin{lemma}\label{lem:subsolutions}
 Let Assumptions \ref{ass:f} and \ref{ass:k} hold  and let $0\le J(t,y)\le K(t,y)$ be measurable.   Assume that $J(\cdot, y)$ has a uniform mean $\overline{J}(y):= \av{J(\cdot, y)}$,
  for almost every $y \in \R$. Define 
  \begin{equation*}
  \mathcal{I}_{J} :=\left\{\lambda\in\R:
  \int_\R \| J(\cdot, y)\|_{L^\infty} e^{\lambda y}\dd y<\infty\right\}. 
  \end{equation*}
   Let    $Q=[p,q]\subset\subset   {\rm Int} \, \mathcal{I}_{J}$.  
  Set
  $$
  P(\lambda) :=\int_\R\overline{J} (y)e^{\lambda y}\dd y-  \av{k_0}+\av r.
  $$
  Assume 
  \begin{equation}\label{eq:negative-gap}
  \max_{\lambda\in Q}\{\lambda P'(\lambda)-P(\lambda)\}<0.
  \end{equation}
  Then, for all  sufficiently large $B>0$ and  sufficiently large $l>B$, there are
  functions $a_\lambda\in W^{1,\infty}(0, \infty)$ uniformly bounded  
  and  a constant  $\kappa_0>0$ such that,
  for $0<\kappa\le\kappa_0$, the functions
  \begin{align}
  C_\lambda(t)&:=\frac{2l}{\pi}   \int_{-B}^B
  J(t,y)e^{\lambda y}\sin\left( \frac{\pi}{2l}   y \right) \dd y,
   \qquad
  X_\lambda(t):=\int_0^t C_\lambda(s)\dd s,\label{eq:centers}\\[2mm]
   \phi_\lambda(t,x)&:=
  \begin{cases}
  \displaystyle \kappa e^{a_\lambda(t)}e^{-\lambda (x-X_\lambda(t) )}\cos\left(\frac{\pi}{2l} (x-X_\lambda(t) ) \right),
  &  \left|  x-X_\lambda(t) \right|<l,\\
  \displaystyle 0,&\left|  x-X_\lambda(t) \right|  \ge l,
  \end{cases}\label{eq:phi}
  \end{align}
  satisfy,  for a.e. $ t>0 \text{  and  }  \left|  x-X_\lambda(t) \right|<l,$
  \begin{equation}\label{eq:diffineq}
  \partial_t \phi_\lambda \leq  \int_{\R} J(t, x-y) \phi_\lambda(t,y) \dd y-  k_0(t)\phi_\lambda   + \phi_\lambda f(t,\phi_\lambda),   
  \end{equation}
  with $\phi_\lambda\in [0,1]$. 
  \end{lemma}

  \begin{proof}
  Set
  $$
  g_\lambda(t)=\lambda\int_\R yJ(t,y)e^{\lambda y}\dd y
  -\int_\R J(t,y)e^{\lambda y}\dd y+k_0(t)-r(t).
  $$
  Its uniform mean is 
   $$\av{g_\lambda}= \lambda P'(\lambda)-P(\lambda).$$
   Since  $Q\subset\subset {\rm Int} \mathcal{I}_J$, we have 
   $$  \int_{\R}  \sup_{ \lambda\in Q}   \{(\lambda y -1)e^{\lambda y} \} \| J(\cdot, y) \|_{L^\infty(0, \infty)} \dd y < \infty. $$
 By Lebesgue dominated convergence theorem, we know 
 $$
   \lim_{T\to\infty}
   \sup_{\lambda\in [p,q]}\sup_{t\ge0}
   \left|
   \frac1T\int_t^{t+T}g_\lambda(s)\dd s
   -\langle g_\lambda\rangle
   \right|=0,
   $$
   It follows from \eqref{eq:negative-gap} and Lemma~\ref{lem:means}  that  there are uniformly bounded $a_\lambda \in W^{1, \infty} (0, \infty)$ and $\theta>0$  such that 
  $$ \sup_{ \lambda\in [p,q]}  \esssup_{ t\geq 0} \left\{  a_\lambda'(t)+g_\lambda(t) \right\} \le -3\theta. $$ 
  We define 
  \begin{equation*}
  J_B(t,y) =
  \begin{cases}
  J(t,y),  & y\in [-B, B], t>0, \\
  0,   & otherwise. 
  \end{cases}
  \end{equation*}
  Since 
  \begin{equation}\label{limitcos}
   \lim_{l\to +\infty} \frac{2l}{\pi} \sin\left( \frac{\pi}{2l} y\right)  = y \text{ and } \lim_{l\to +\infty} \cos\left( \frac{\pi}{2l} y\right) = 1    \quad  \text{uniformly in }    y \in [-B, B],   
  \end{equation}
   we can choose $B, l$ large enough with $l>B$  such that 
   \begin{equation}\label{eq:residual}
   \sup_{ \lambda\in [p,q]}  \esssup_{ t\geq 0}   \left\{  a_\lambda'(t)+\lambda C_\lambda(t)
    -\int_{\R} J_B (t,y)e^{\lambda y}\cos\left(\frac{\pi}{2l} y\right) \dd y+k_0(t)  -  r(t)  \right\}
   \leq  -2\theta.
   \end{equation}
   Note that 
   \begin{align*}
     \int_{\R}J(t,x-y) \phi_{\lambda} (t,y) \dd y&= \int_{-l+X_\lambda}^{l+X_\lambda} J(t,x-y) \kappa  e^{a_\lambda (t)} e^{-\lambda (y- X_\lambda ) } \cos \left( \frac{\pi }{2l} (y-X_\lambda)   \right) \dd y    \\
      &=\int_{-l}^{l} J(t,x-z-X_\lambda) \kappa    e^{a_\lambda (t)} e^{-\lambda z} \cos \left( \frac{\pi }{2l} z \right) \dd z  \\
     & \geq \int_{-l}^{l} J_B(t, x-z-X_\lambda)\kappa   e^{a_\lambda(t)}  e^{-\lambda z}  \cos  \left( \frac{\pi }{2l}z \right) \dd z \\
     &\geq \int_{\R}J_B(t,x-z-X_\lambda) \kappa  e^{a_\lambda(t)}  e^{-\lambda z}  \cos  \left( \frac{\pi }{2l} z  \right) \dd z
   \end{align*} 
   The last inequality can be derived from $\cos( \frac{\pi z}{2l}) <0$ for $B< l<|z|<2l$ and $J_B(t,x-z-X_\lambda)=0 $ for  $|z| >2l $ and $|x-X_\lambda| <l$.    This estimate follows the construction in  \cite{diekmann1979run, ducrot2021generalized}.
   We observe
  \begin{equation*}
  \begin{split}
   & \int_{\R}J_B(t,x-z-X_\lambda) e^{-\lambda z}  \cos  \left( \frac{\pi }{2l} z  \right) \dd z   
  = \int_{-B}^{B}  J(t,y)    e^{-\lambda  (x-y-X_\lambda)} \cos\left( \frac{\pi}{2l} (x-y-X_\lambda)  \right)  \dd y\\[2mm]
   &= \int_{-B}^{B}  J(t,y)    e^{-\lambda  (x-y-X_\lambda)} \bigg[ \cos\left( \frac{\pi}{2l} (x -X_\lambda)  \right)\cos\left( \frac{\pi}{2l} y \right)  +\sin\left( \frac{\pi}{2l} (x -X_\lambda)  \right) \sin \left( \frac{\pi}{2l}  y   \right)     \bigg]  \dd y.
  \end{split}
  \end{equation*}
   Differentiating at almost every $t>0$ with  $|x-X_\lambda(t)|<l$, we obtain 
     \begin{align*}
     &\partial_t\phi_\lambda-\int_{\R}J(t,x-y)\phi_\lambda (t,y) \dd y
                        + k_0(t)\phi_\lambda -r(t) \phi_\lambda\\
     &\quad\le \phi_\lambda
     \left( a_\lambda'+\lambda C_\lambda
     -\int_{-B}^B J(t,y)e^{\lambda y}\cos\left( \frac{\pi}{2l} y \right)\dd y+k_0-r\right)\\
     &\qquad+\kappa e^{a_\lambda-\lambda (x-X_\lambda)}\sin\left( \frac{\pi}{2l} (x-X_\lambda)\right)
     \left(\frac{\pi}{2l} C_\lambda-
     \int_{-B}^B J(t,y)e^{\lambda y}\sin\left( \frac{\pi}{2l} y \right)\dd y\right) \\
     &\quad \leq -2\theta \phi_{\lambda}. 
     \end{align*}
  The last inequality follows from the definition of $C_\lambda$ and \eqref{eq:residual}. 
 By   Assumption~\ref{ass:f},    $f(t,\phi_{\lambda}) \geq r(t) - L \phi_{\lambda}$. 
  We choose $\kappa_0>0$  small  enough such that   
   $$
  \kappa_0  \exp\left\{  \sup_{\lambda\in Q}\|a_\lambda\|_\infty
  +l\max_{\lambda\in Q}|\lambda| \right\} 
  \le\min\{1,\theta/L\}.
  $$
  Hence,  we have $(-2\theta+L\phi_\lambda)\phi_\lambda\leq 0$.   The function $\phi_{\lambda}$ satisfies \eqref{eq:diffineq}.     The proof is complete.
   
  \end{proof}
   
   The subsolutions in Lemma~\ref{lem:subsolutions}  are  supported on   $[X_\lambda (t) -l, X_\lambda +l]$.   These support sets have a fixed length $2l$ moving with speed  $C_\lambda(t)$. From \eqref{limitcos}, we know $\av{C_\lambda} \to P'(\lambda) $  uniformly on $Q$   as $B\to \infty$ and $l \to \infty$. 
    The next lemma uses this relation  to obtain a  positive lower bound on an expanding interval.

   \begin{lemma} \label{lem:cover}
   Let  assumptions in Lemma~\ref{lem:subsolutions}   hold on $[p,q]$, with  $p<q$. Let $u$ solves \eqref{Pb} and $u \not \equiv 0$.
   If $P'(p)<a<b<P'(q)$, then
   \begin{equation}\label{eq:positive-lower}
   \liminf_{t\to+\infty}\inf_{at\le x\le bt}u(t,x)>0.
   \end{equation}
   \end{lemma}
   
   \begin{proof}
   Choose $d_1,d_2$ such that
   $$
   P'(p)<d_1<a<b<d_2<P'(q).
   $$
   By Lemma~\ref{lem:subsolutions}, choose $B$ and $l>B$ sufficiently large  so that  
   $$
   \av{C_p}<d_1,\qquad \av{C_q}>d_2.
   $$
  Let $a_\lambda$ and  $\kappa_0>0$ be given by Lemma~\ref{lem:subsolutions}, and set
   $$
   A:=\sup_{\lambda\in [p,q]}\|a_\lambda\|_{L^\infty(0,\infty)}<\infty,
   \qquad
   \Lambda:=\max \{|p|,|q| \}.
   $$
  It follows from   Lemma~\ref{lem:positive} that  there exist   $T_0>0$ and $x_0\in\R$ such that
   $$
   \delta:=\min_{x_0-l\le x\le x_0+ l}u(T_0,x)>0.
   $$
   We use $u(T_0,\cdot)$ as the initial value.  Let
   $$
   \widehat X_\lambda(s)
   :=\int_0^s C_\lambda(T_0+\tau)\dd\tau, \quad  s\ge 0.
   $$
  Define
   $$
   \psi_\lambda(s,x)=
   \begin{cases}
   \kappa e^{a_\lambda(T_0+s)}
   e^{-\lambda z}\cos(\frac{\pi}{2l} z),& |z|<l,\\
   0,& |z|\ge l,
   \end{cases}
   \qquad
   z=x-x_0-\widehat X_\lambda(s).
   $$
 For almost every $s>0$  with $ \left|  x-x_0-\widehat{X}_\lambda(s) \right|<l$, $\psi_\lambda$  satisfies
   \begin{equation*} 
   \partial_s \psi_\lambda \leq  \int_{\R} J(T_0+ s, x-y) \psi_\lambda(s,y) \dd y - k_0(T_0+s) \psi_\lambda   + \psi_\lambda f(T_0+s,\psi_\lambda). 
   \end{equation*}
Choose $
   0<\kappa\le
   \min\{\kappa_0,   \delta e^{-A-l\Lambda}\}.
   $
   By $\widehat X_\lambda(0)=0$, we have 
   \begin{align*}
    & 0\le\psi_\lambda(0,x)  \le u(T_0,x),   \quad \forall |x-x_0|\leq l, \\
    & \psi_\lambda(0, x) =0 \le u(T_0,x),  \quad \forall |x-x_0|>l.
   \end{align*}
   Since $J\le K$ and $\psi_\lambda \geq 0$, each $\psi_\lambda$ is a subsolution of the time shifted equation \eqref{Pb}.   Applying  Lemma~\ref{lem:comparison},   gives 
   $$
   u(T_0+s,x)\ge\psi_\lambda(s,x), 
   \qquad   \forall  s\ge0,\;  x\in\R,\;  \lambda\in [p,q].
   $$
   For each subsolution, we notice 
   $$
   u(T_0+s,x_0+\widehat X_\lambda(s))  \ge\kappa e^{a_\lambda(T_0+s)} \geq \kappa e^{-A}=:m>0,  \quad \forall s\ge 0, \lambda \in [p, q].
   $$
   The constant $m$ is independent of $s$ and $\lambda$.

   We claim that the map $\lambda\mapsto\widehat X_\lambda(s)$
    is continuous  for every fixed $s>0$. Indeed, since 
    $$J(t,y)\leq K(t,y) \leq \|K(\cdot, y)\|_{L^\infty(0, \infty)} \quad  \text{ and } \quad  \frac{2l}{\pi} \left| \sin\left( \frac{\pi}{2l} y\right)\right| \leq |y|, $$
   we obtain 
   \begin{equation*}
      |\widehat X_\lambda(s)-\widehat X_\mu(s)|
      \le s|\lambda-\mu|
      \int_{-B}^B k(y)y^2e^{\Lambda|y|}\dd y, \quad \forall \lambda, \mu \in [p, q].
   \end{equation*}
   The  integral is finite.  Thus, the map $\lambda\mapsto \widehat{X}_{\lambda}$ is continuous. 
  The definition of  uniform mean yields
   $$
   \frac{\widehat X_p(s)}s\rightarrow\av{C_p}<d_1,
   \qquad
   \frac{\widehat X_q(s)}s\rightarrow\av{C_q}>d_2
   \qquad  \text{ as }  \;\; s\to+\infty.
   $$
   For all sufficiently large $s$, we obtain
   $$
   \widehat X_p(s)<d_1s<d_2s<\widehat X_q(s).
   $$
   For each $s$ large enough and  $x\in[x_0+d_1s,x_0+d_2s]$, the intermediate value theorem  implies that there  exist  $\lambda \in [p,q]$  such that
   $$ 
   \widehat X_\lambda(s)=x-x_0.
   $$
   Hence, we have 
   $$
   u(T_0+s,x)\ge m
   \qquad
    \forall   x_0+d_1s\le x\le x_0+d_2s.
   $$

   Finally, let $t=T_0+s$.  We note that 
   $$
   at-(x_0+d_1s)=(a-d_1)s+aT_0-x_0>0, \quad   \forall s>\frac{2(x_0- aT_0)}{a-d_1},
   $$
   $$
   x_0+d_2s-bt=(d_2-b)s+x_0-bT_0>0, \qquad \forall s>\frac{2(bT_0-x_0)}{d_2 -b} .
   $$
  Thus, for all sufficiently large $t$,
   $$
   [at,bt]\subset[x_0+d_1s,x_0+d_2s]
   $$
    and $u(t,x)\ge m$ for  $x\in [at,bt]$.
   This proves \eqref{eq:positive-lower}.  The proof is complete.

   \end{proof}

  \subsection{Convergence to one}
  
We show that a positive lower bound  on an expanding interval implies convergence to $1$ on smaller spreading interval.  The proof is inspired by \cite{liang2020jfa}.

   \begin{lemma}\label{lem:upgrade}
   Let Assumptions~\ref{ass:f} and~\ref{ass:k} hold. 
   Let $a_0<a<b<b_0$. Suppose that
   $$
   \liminf_{t\to+\infty}
   \inf_{a_0t\le x\le b_0t} u(t,x)>0.
   $$
   Then
   $$
   \lim_{t\to +\infty}\sup_{at\le x\le bt}|u(t,x)-1|=0.
   $$
   \end{lemma}
   \begin{proof}
   We first give a finite-time estimate.  
   Set $\gamma=\eta/2$, where $\eta$ is given by
    \eqref{eq:envelope}, and for $z\in\R$ define
    $$
    \|w\|_z:=\sup_{x\in\R}e^{-\gamma|x-z|}|w(x)|.
    $$
    Let $v_1,v_2$ be two solutions of \eqref{Pb}, both starting
     at time $\tau\ge0$. Denote   $w=v_1-v_2$. 
      The integral equations for $v_1, v_2$ give
    \begin{equation*}
     \begin{aligned}
     w(t,x)=w(\tau,x)
     +\int_\tau^t\bigg[
     &\int_\R K(s,x- y)w(s,y)\dd y-k_0(s)w(s,x)\\
     &\quad +v_1(s,x) f(s,v_1(s,x))-v_2(s,x) f(s,v_2(s,x))
     \bigg]\dd s.
     \end{aligned}
     \end{equation*}
  We note that 
     \begin{equation*}
       \begin{aligned}
        e^{-\gamma|x-z|}
        \left|\int_\R K(s,x-y)w(s, y)\dd y\right|
        &\le\int_\R K(s,y)
        e^{\gamma(|x-y-z|-|x-z|)}\dd y\,\|w(s,\cdot)\|_z\\
        &\le M_\gamma\|w(s,\cdot)\|_z,
        \end{aligned}
     \end{equation*}
     where
     $$
     M_\gamma:=\int_\R k(y)e^{\gamma|y|}\dd y<\infty.
     $$
     Multiplying the integral equation   by $e^{-\gamma |x-z| }$ and taking the supremum over  $x$, we obtain
     $$
     \|w(t,\cdot)\|_z
     \le\|w(\tau,\cdot)\|_z
     +C\int_\tau^t\|w(s,\cdot)\|_z\dd s,
     \qquad
     C:=M_\gamma+\|k_0\|_{L^\infty(0,\infty)}+2L.
     $$
     Gronwall's inequality  yields
  \begin{equation} \label{eq:locality}
     |v_1(t,z)-v_2(t,z)|
     \le\|w(t,\cdot)\|_z
     \le e^{C(t-\tau)}
     \|v_1(\tau,\cdot)-v_2(\tau,\cdot)\|_z
     \qquad \forall  t\ge\tau.
  \end{equation}
     The constant $C$ is independent of $\tau,z$, and of the two solutions.
     
  Let  
    $$
     0<\theta<
     \min\left\{a-a_0,\ b_0-b,\ \frac{b_0-a_0}{2}\right\}.
     $$  
  If $v_1(\tau, x) = v_2(\tau, x)$ on  $  I_\tau=[a_0\tau,b_0\tau]$, 
  and  $z\in[(a_0+\theta)\tau,(b_0-\theta)\tau]$, 
  then    
  \begin{equation*}
   \|v_1(\tau,\cdot)-v_2(\tau,\cdot)\|_z
   =\sup_{x\notin I_\tau}
   e^{-\gamma|x-z|}|v_1(\tau,x)-v_2(\tau,x)|  \le e^{-\gamma\theta\tau}.
  \end{equation*}
   Thus, for every fixed $T>0$,
   $$
   |v_1(\tau+T,z)-v_2(\tau+T,z)|
   \le e^{CT-\gamma\theta\tau}
   \qquad   
    \forall  z\in  [(a_0+\theta)\tau,  (b_0-\theta)\tau].
   $$

By the assumption 
    $$
    \liminf_{t\to+\infty}
    \inf_{a_0t\le x\le b_0t} u(t,x)>0, 
    $$
    we choose $\alpha\in (0,1)$ and $\tau_1>0$
   such that
    $$
    u(\tau,x)\ge\alpha
    \qquad  \forall   \tau\ge\tau_1,  \;    a_0\tau\le x\le b_0\tau.
    $$
   Set
    $$
    h(s):=\essinf_{t\ge0}f(t,s),\qquad 0\le s\le1.
    $$
    The uniform Lipschitz bound for $f$ gives
    $|h(s_1)-h(s_2)|\le L|s_1-s_2|$. Assumption~\ref{ass:f} and
    \eqref{eq:f-bounds} also yield $h(s)>0$ for  $s\in [0,1)$ and $h(1)=0$.
    Let $q$  be the solution of 
    $$
    q'(s)=q(s)h(q(s)),\qquad q(0)=\alpha.
    $$
    Then $\alpha\le q(s)\le1$ and
    $$  \lim_{t\to+\infty} q(t) =1.$$

   For each $\tau\ge0$, let $\rho^\tau$ solve the nonautonomous 
   equation
   $$
   (\rho^\tau)'(s)
   =\rho^\tau(s)f(\tau+s,\rho^\tau(s)),
   \qquad \rho^\tau(0)=\alpha.
   $$
   By the  comparison principle, we have 
   $$ q(s) \leq   \rho^\tau(s) \le 1,  \qquad  \forall s\geq 0, \tau \geq 0$$
   Note that $V(\tau+s, x):= \rho^\tau(s)$ solves 
   \begin{equation*}
  \begin{cases}
  \displaystyle \partial_s V= \int_{\R} K(s+\tau, x-y) [V(\tau+s,y) -V(\tau+s,x) ] \dd y  + V f(\tau+s,V ), &s>0,  \;  x\in\R, \\
  \displaystyle  V(\tau, x) = \alpha, & x\in \R.
  \end{cases}
   \end{equation*}
    Fix $\sigma\in(0,1/2)$. Choose $T>0$ such  that $q(T)>1-\sigma$.
  Hence, $  V(\tau+T,x) >1-\sigma$ for all $\tau \geq 0$ and $x\in \R$.

    For $\tau\ge\tau_1$, let $v^\tau (\tau+s, x)$ be the solution of 
     \begin{equation*}
    \begin{cases}
    \displaystyle \partial_s v^\tau = \int_{\R} K(\tau+s, x-y) [v^\tau(\tau+s,y) -v^\tau(\tau+s,x) ] \dd y  + v^\tau  f(\tau+s,v^\tau), &  s>0,  \;  x\in\R, \\
    \displaystyle  v^\tau ( \tau, x) = \min\{\alpha,  u(\tau, x) \}, & x\in \R.
    \end{cases}
     \end{equation*}
applying comparison principle yields $u\geq v^{\tau}$.
  By \eqref{eq:locality},   for every $\tau\ge\tau_1$ and
    $z\in[(a_0+\theta)\tau,(b_0-\theta)\tau]$, we obtain
   \begin{equation*}
     u(\tau+T,z)
     \ge v^\tau(\tau+T,z)  \ge V (\tau+T,z)-e^{CT-\gamma\theta\tau} \ge 1-\sigma-e^{CT-\gamma\theta\tau}.
   \end{equation*}
   Fix $T>0$ as above. We choose  $\tau_2\geq \tau_1$ sufficiently large  such   that 
   $e^{CT-\gamma\theta\tau}\le\sigma $    for all   $\tau\ge\tau_2. $ 
   Then we obtain 
    $$
    u(\tau+T,z)\ge1-2\sigma
    \qquad
   \forall  \tau\ge\tau_2,  \;  \forall z\in [(a_0+\theta)\tau,  (b_0-\theta)\tau].
    $$
  Set $\tau= t- T$.   For sufficiently   large $t$,    we have  $\tau\ge \tau_2$ and 
      $$
      [at,bt]\subset
      [(a_0+\theta)(t-T),(b_0-\theta)(t-T)]. 
      $$
    Hence we have 
    $$1\ge u(t,x) \geq 1-2\sigma,   \quad \forall x\in [at, bt], t\geq T+\tau_2.$$
    By the arbitrariness of $\sigma$, we obtain  the desired conclusion. The proof is complete. 
   
   \end{proof}

  \subsection{Lower estimate in Cases~(i)--(iii)}\label{sec:proof2}

 For lower estimate, the key point in the construction of subsolutions is the   condition \eqref{eq:negative-gap}.   Now we prove $c_+^*+c_-^*>0$ and \eqref{eq:lower-one} under Cases $(i)$-$(iii)$ in Proposition \ref{prop:classification}.  

%
%
%

    Since $ \int_{\R} \overline K(y) \dd y >0$, we have 
    $$
    H''(\lambda)
    =\int_\R y^2\overline K(y)e^{\lambda y}\dd y>0
    \qquad(i_-<\lambda<i_+).
    $$
    Recall that
    $$
    g(\lambda)=\lambda H'(\lambda)-H(\lambda),
    \qquad
    g(0)=-\av r<0,
    \qquad
    g'(\lambda)=\lambda H''(\lambda).
    $$
    Thus $g$ is  increasing on $(0,i_+)$ and   decreasing
    on $(i_-,0)$.
    
    We first consider the positive direction. 
    Let $\rho_+\in (0, i_+]$ be the positive minimizing parameter of $H(\lambda)/\lambda$.  Precisely, $\rho_+ \in (0, i_+)$ in Case~$(i)$, $\rho_+=+\infty$ in  Case~$(ii)$ and  $\rho_+=i_+$ in  Case~$(iii)$. 
    From Proposition~\ref{prop:classification}, we know that 
    $$  g(\lambda) <0  \quad \forall \lambda \in [0, \rho_+)\quad  \text{ and } \quad \lim_{\lambda \to \rho_+}H'(\lambda) = c^*_+.$$

    For the negative direction, we use the reflection argument through
    $$
    \widetilde K(t,y):=K(t,-y),\qquad
    \widetilde H(s):=H(-s).
    $$
   Define $\rho_- \in [i_-, 0) $  in the same way for $H(\lambda) / (-\lambda)$, allowing $\rho_-=-\infty$.
    Proposition~\ref{prop:classification} yields 
     $$  g(\lambda) <0  \quad \forall \lambda \in   (\rho_-, 0]  \quad  \text{ and } \quad \lim_{\lambda \to \rho_-}H'(\lambda) = - c^*_-.$$
    Hence,  we obtain 
    \begin{equation}\label{eq:direct-limits}
    \begin{aligned}
    &\qquad  g(\lambda)<0 \qquad \quad \forall  \rho_-<\lambda<\rho_+,\\
    & \lim_{\lambda\to \rho_-}H'(\lambda)=-c_-^*,
    \qquad
    \lim_{\lambda\to \rho_+}H'(\lambda)=c_+^*.
    \end{aligned}
    \end{equation}
    Let us  choose finite $\lambda_1 \in (\rho_- ,0)$ and $\lambda_2\in (0, \rho_+)$.   
    Since $H''>0$ on $(i_-, i_+)$, we have 
    $$
    -c_-^*<H'(\lambda_1)<H'(\lambda_2)<c_+^*.
    $$
    Thus $c^*_+ + c_-^*>0$.

    We now prove \eqref{eq:lower-one}. 
    Fix $a,b$ with $-c_-^*<a<b<c_+^*$.  Let  $a_0$ and $b_0$  satisfy 
    $$
    -c_-^*<a_0<a<b<b_0<c_+^*.
    $$
    By \eqref{eq:direct-limits}, we can choose  $p,q$  
    such that
    $$
    \rho_-<p<0<q<\rho_+,
    \qquad
    H'(p)<a_0<b_0<H'(q).
    $$
    Note that 
    $$  Q:=[p, q] \subset \subset (i_-, i_+). $$
    Since $g$ is continuous and  
    $$
     \max_{\lambda\in Q}g(\lambda)<0,
    $$
    we can apply Lemmas~\ref{lem:subsolutions} and \ref{lem:cover}  with $J=K$ and $P=H$ to obtain 
    \begin{equation*}
    \liminf_{t\to+\infty} \inf_{a_0 t\le x\le b_0 t} u(t,x)>0.
    \end{equation*}
    Then Lemma~\ref{lem:upgrade} yields
    $$
    \lim_{t\to\infty}
    \sup_{at\le x\le bt}|u(t,x)-1|=0.
    $$
    This proves \eqref{eq:lower-one}.

  \subsection{Lower estimate in Case~(iv)}\label{sect:truncate}

 The preceding argument  cannot  show the lower estimate of spreading speed in Case $(iv)$, since 
$$
\lim_{\lambda\uparrow  i_+ } H'(\lambda) =H'(i_+)<\frac{H(i_+)}{i_+}=c_+^*.
$$
 We address this issue by truncated kernels under additional  Assumption~\ref{ass:endpoint}. 
 The truncation argument is inspired by  \cite[Section 5.3]{liang2020jfa}.  
 
  For $n\in\mathbb N$, let 
\begin{equation}\label{eq:truncation}
K_n(t,y)=K(t,y)\mathbf1_{[-n,n]}(y),\qquad
H_n(\lambda)=\int_{-n}^n\overline K(y)e^{\lambda y}\dd y-\av{k_0}+\av r.
\end{equation}
Note that 
$$ \left\{  \lambda \in \R:  \int_{\R}K_n(t,y) e^{\lambda y} \dd y < \infty  \right\} = (-\infty, +\infty). $$
In particular, for sufficiently large $n$,  
$$ H_n(0)=\av r-\int_{|y|>n}\overline{K}(y)\dd y >0 \quad \text{ and }  \quad \int_{-n}^{n} \overline{K} (y) \dd y>0. $$
Define
\begin{equation}\label{eq:truncated-speeds}
c_{n,+}:=\inf_{\lambda>0}\frac{H_n(\lambda)}\lambda,\qquad
c_{n,-}:=\inf_{\lambda<0}\frac{H_n(\lambda)}{-\lambda}.
\end{equation}

\begin{lemma}\label{lem:truncation-convergence}
Under the assumptions  of Theorem~\ref{thm:spreading},
we have 
\begin{equation}\label{eq:speed-convergence}
c_{n,+}\uparrow c_+^* \quad   \text{ and } \quad c_{n,-}\uparrow c_-^*  \quad \text{ as } n \to +\infty.
\end{equation}
In case (iv) of Proposition~\ref{prop:classification} in the positive direction,  for all sufficiently large $n$, there is   a unique  $p_n>0$ such that 
$$   c_{n,+} = \frac{H_n (p_n)}{p_n} = H'_n(p_n),      \qquad p_n>i_+, \qquad  c_{n,+}<c_+^*, $$
and 
  $p_n\to i_+$ as  $n\to \infty$.
\end{lemma}
\begin{proof}
We only prove the positive direction. 
The other side  follows  by reflection argument.  Since $H_n$ increases with $n$ and $H_n\le H$ on $(0,i_+)$,  we observe that 
$c_{n,+}$ is nondecreasing and $c_{n,+}\le c_+^*$.   
By convexity of $H_n(\lambda)$,  we  know 
$$H_n(\lambda) \geq H_n (0) + H'_n(0) \lambda.   $$
We can choose $n\geq N$ large enough such that $H_n(0) \geq \av{r} /2$.   There exists $M_1>0$ independent of $n$ such that  $| H_n'(0) | \leq M_1$.  Thus we have 
$$ \frac{H_n(\lambda)}{\lambda} \geq \frac{\av{r}}{2\lambda}  - M_1, \quad \forall n\geq N.  $$
Choose $\delta>0$ small enough  such that  
$$\frac{\av{r}}{2\lambda}  - M_1  >c_+^*+1   \qquad \forall 0 < \lambda <\delta.$$
Thus, for all $n \ge N$,   the infimum point of $c_{n, +}$ may be restricted to   $\lambda\ge \delta$.

We first  consider  case $(iv)$.  Let $i_+<  \infty$ and $H(i_+)/ i_+ > H'(i_+)$.  
By \eqref{eq:tail-condition}, for all large $n$,  we have  $m:= \int_{a}^{b}\overline{K}(y)\dd y>0$ with $0<a<b$, and 
$$
H_n(\lambda)\ge m e^{a\lambda}- \av{k_0}+\av r.
$$
We notice
$$\frac{H_n (\lambda)}{\lambda} \to + \infty   \quad  \text{  as  } \lambda \to 0 \quad \text{ and } \quad   \frac{H_n (\lambda)}{\lambda} \to  + \infty   \quad  \text{  as  } \lambda \to +\infty, $$
uniformly in sufficiently large $n$.  
There exist  a fixed constant $R$ and  $p_n\in [\delta, R]$ for all large $n$, such that 
\begin{equation*}
c_{n,+}= \frac{H_n(p_n)}{p_n}= H_n'(p_n). 
\end{equation*}
In addition,  $p_n$ is the  unique minimization parameter of $H_n(\lambda)/\lambda$,  since $\lambda H_n'(\lambda) -H_n(\lambda)$ is increasing on $(0, \infty)$.  
We have   
$$  i_+ H_n'(i_+) - H_n(i_+) \to i_+ H'(i_+) -H(i_+) <0  \quad \text{ as } n \to +\infty.  $$
This gives $  i_+ H_n'(i_+) - H_n(i_+)  <0 $ for all large $n$.   Since $\lambda H_n'(\lambda)  - H_n(\lambda)$ is increasing on $(0, \infty)$  and $  p_n H_n'(p_n) - H_n(p_n)=0$, we have $p_n>i_+$ for all large $n$.  Choose subsequence  $p_n \to p$ as $n \to \infty$.  
 Next we prove $p= i_+$. If $p> i_+$,  Fatou lemma and
 \eqref{eq:tail-condition} give
$$
 \liminf_{n \to \infty}  \int_{-n}^n\overline K(y)e^{p_n y}\dd y
 \ge\int_\R\overline K(y)e^{py}\dd y=+\infty.
 $$
This contradicts $H_n(p_n) = p_n c_{n, +} \leq |c_+^*| R$.   
 Thus we obtain  $p= i_+$.
 Applying Fatou lemma, one has 
 $$
 \liminf_{n \to \infty} H_n(p_n)\ge H(i_+),\qquad
 \liminf_{n\to \infty}  c_{n,+}\ge \frac{H(i_+) }{i_+} =c_+^*.
 $$
 Assumption \eqref{eq:tail-condition} implies 
 $$   \int_{-n}^{n} \overline{K}(y) e^{i_+ y} \dd y < \int_{\R}  \overline{K}(y) e^{i_+ y} \dd y. $$
 Since 
 $$c_{n,+} \leq \frac{H_n(i_+)}{i_+}  <  \frac{H(i_+)}{i_+} = c_+^*,   $$
 we have $c_{n,+} \uparrow c^*_+$ as $n \to +\infty$.
 
 To consider the two directions belong to different cases, we also prove \eqref{eq:speed-convergence} in Cases~$(i)$-$(iii)$.
 Let $p$ be the  interior minimizing  parameter, or $p= i_+ <+\infty$.  Note that $pH'(p)=H(p)$, and 
\begin{equation*}
H_n(p)\to H(p),\qquad H_n'(p)\to H'(p),\qquad  \text{ as } n \to +\infty.
\end{equation*}
By the convexity of $H_n$, we have 
$$
\frac{H_n(\lambda)}\lambda\ge H_n'(p)
 +\frac{H_n(p)-pH_n'(p)}\lambda
\ge H_n'(p)-\frac{|H_n(p)-pH_n'(p)|}{\delta},  \quad 
\forall \lambda \geq \delta.
$$
Taking infima and letting $n \to +\infty$, we obtain 
 $$
  \liminf_{n \to+\infty} c_{n, +}  \geq c^*_{+}.  
 $$

In case~$(ii)$, we have $\overline K=0$ on $(0,+\infty)$.   From the proof of Proposition~\ref{prop:classification}, we note that 
$$g(\lambda) = \lambda H'(\lambda) -H(\lambda) \leq 0, \quad \forall \lambda>0$$
and 
$$ H(\lambda) \to -\av{k_0} +\av{r} \text{ and   }  \lambda H'(\lambda) \to 0, \quad \text{ as } \lambda \to +\infty.  $$
Consequently,  $   \av{r}  -\av{k_0} \geq 0$ and 
$$
\frac{H_n(\lambda)}{\lambda} \geq \frac{  \av{r}  -\av{k_0}  }{ \lambda}. 
$$
Since $H_n(\lambda)$ is bounded for each fixed $n$, we have $c_{n, +} = c_+^*=0$.
 We complete the proof.
   
\end{proof} 

\begin{remark}
In  case $(iv)$, Assumption~\ref{ass:endpoint} is also necessary for 
 $ c_{n,+} \to c_+^*$.  Suppose that  the condition fails.  There exists $\mu>i_+$ with $ \int_{\R}\overline{K}(y) e^{\mu y}\dd y <\infty.  $  For this remark, we extend $H$ to $(i_-, \mu)$  in \eqref{eq:H}. Note that 
$$\left(  \frac{H(\lambda)}{\lambda} \right)'\bigg|_{\lambda=i_+} <0.$$
Hence, there exists   $\nu \in (i_+, \mu) $ such that 
$$
\frac{H(\nu)}{\nu} < \frac{H(i_+)}{i_+} = c^*_+.
$$
For all $n$
$$ \frac{H(\nu) }{\nu } \geq  \frac{H_n (\nu) }{\nu }\geq c_{n, +}. $$
Thus $c_{n, +} $ cannot converge to $c_+^*$.
 
\end{remark}

We now  complete the proof of Theorem \ref{thm:spreading}.
\begin{proof} 
We first  prove $c^*_+ + c_-^* >0$.  
Let $p_n \in [-\infty, 0) $ and $q_n \in (0, +\infty]$ satisfy 
\begin{equation*}
c_{n, -} = \frac{H_n (p_n)}{-p_n} =  -H'_n(p_n) \text{  and  } c_{n, +} = \frac{H_n (q_n)}{q_n}=H_n'(q_n).
\end{equation*}
Since $\int_{\R} \overline{K} (y) \dd y>0$,  we cannot have both $p_n=-\infty$ and $q_n=+\infty$.
The increasing property of $H'_n$ implies $c_{n,+}+c_{n,-} = H_n'(q_n) -H_n'(p_n)>0$.
Since   $c_{n, \pm} \leq c_{\pm}^*$,  we obtain $c^*_+ + c_-^* >0$.

Fix  $-c_-^*<a_0<a<b<b_0<c_+^*$.   By Lemma~\ref{lem:truncation-convergence},
 we  choose $n$ sufficiently large that
 $$
-c_{n,-}<a_0<b_0<c_{n,+}.
$$
Let $n$ be fixed.   We choose finite $p, q\in (p_n, q_n)$ with $p<q$ such that 
$$
  H_n'(p) < a_0 <b_0 <H_n'(q) \quad \text{ and } \quad 
  \max_{\lambda\in[p,q]} \left\{\lambda H_n'(\lambda)-H_n(\lambda) \right \}<0.
$$
Apply Lemmas \ref{lem:subsolutions} and \ref{lem:cover} with $J=K_n$ and $P= H_n$.   Note that  $\phi_\lambda$ satisfies
\begin{equation*}
\begin{aligned}
\partial_t \phi_\lambda &\leq  \int_{-n}^{n} K(t, y) \phi_\lambda(t,x- y) \dd y-  k_0(t)\phi_\lambda   + \phi_\lambda f(t,\phi_\lambda) \\
& = \int_{\R}K(t,y) [\phi_\lambda(t,x-y) - \phi_{\lambda}(t,x)] \dd y + \phi_\lambda f(t,\phi_\lambda) - \int_{|y|>n} K(t,y) \phi_\lambda(t,x-y) \dd y.
\end{aligned}
\end{equation*}
 Since $\int_{|y|>n} K(t, y) \phi_\lambda(t,x-y) \dd y$ is nonnegative, we know that $\phi_{\lambda}$ is a subsolution of \eqref{Pb}.
Hence, Lemma \ref{lem:cover} gives 
 \begin{equation*}
 \liminf_{t\to+\infty}\inf_{a_0 t\le x\le b_0t}u(t,x)>0.
 \end{equation*}
Applying  Lemma~\ref{lem:upgrade}, we have 
 $$
 \lim_{t\to\infty}
 \sup_{at\le x\le bt}|u(t,x)-1|=0.
$$
This proves  \eqref{eq:lower-one}  for every $-c_-^* <a<b<c_+^*$. 
The proof is  complete. 

\end{proof}

\section{Examples} \label{sec:ex}

We give three examples illustrating the classification in Proposition~\ref{prop:classification} and the effect of time-dependent asymmetric dispersal kernel.

Our first example shows that  the  condition $\int_{\R}\overline{K} (y)\dd y >0$  cannot in general be omitted from  Theorem \ref{thm:spreading}  for the conclusion  $c_+^*+c_-^*>0$.  The solution can converge locally  to   the positive equilibrium when $c_+^*=c_-^*=0$.

 \begin{example}\label{ex00}
 Set
 $$
 J(y):=\mathbf1_{[-1,1]}(y),\qquad
 K(t,y):=e^{-t}J(y),\qquad f(t,u):=1-u.
 $$
 Let $u_0 $ be continuous,  compactly supported,  $0\le u_0\le1$ and
 $u_0\not\equiv0$.  Let $u$ be a solution of 
 \begin{equation}\label{eq:zero-mean-model}
 \begin{cases}
 \displaystyle
 u_t(t,x)=e^{-t}\int_{-1}^{1}
  [u(t,x-y)-u(t,x)]\dd y+u(t,x)(1-u(t,x)),
  &t>0,\ x\in\R,\\[2pt]
 u(0,x)=u_0(x),&x\in\R.
 \end{cases}
 \end{equation}
 Then   one has $\overline{K}= 0$   almost everywhere on $\R$ and 
 \begin{equation}\label{easy}
 \mathcal{I}= \R, \quad H(\lambda) \equiv 1, \quad   c_+^*=c_-^*=0. 
 \end{equation}
 
However,  $u$ satisfies 
 \begin{equation}\label{eq:zero-mean-local}
 \lim_{t\to\infty}\sup_{|x|\le R}|u(t,x)-1|=0
 \qquad\text{for every }R>0.
 \end{equation}
At the same time, 
\begin{equation}\label{eq:zero-mean-linear-upper}
\lim_{t\to\infty}\sup_{|x|\ge ct}u(t,x)=0
\qquad\text{for every }c>0.
\end{equation} 
 
 \end{example}
 \begin{proof}
 Since $\av{e^{-t}}=0$, we have $\overline{K}=0$ and \eqref{easy}.  
Then we prove \eqref{eq:zero-mean-local}.  
Similarly to Lemma~\ref{lem:positive}, since $J$ is positive on $[-1, 1]$, we can show that $u(t,x)>0$ for all $x\in \R$ and $t>0$. 

 Fix $R>0$ and $\eta\in(0,1)$. Choose $T_\eta>0$ such that
 $2e^{-t}\le\eta$ for all $t\ge T_\eta$. Positivity and continuity give
 $$
 \delta_R:=\min_{|x|\le R} u(T_\eta,x)>0.
 $$
 Since $\int_{-1}^{1} e^{-t} u(t,x-y) \dd y \geq 0$, the solution $u(t,x)$ satisfies 
 \begin{equation}\label{eq:zero-mean-ode-inequality}
 u_t(t,x)\ge u(t,x)(1-2e^{-t}-u(t,x))
 \ge u(t,x)(1-\eta-u(t,x))
 \qquad \forall  t\ge T_\eta.
 \end{equation}
 Let $q =q(t)$ solve 
  $$ q'=q(1-\eta-q), \qquad  \text{  with  }  q(T_\eta)=\delta_R.  $$
 For each fixed $x\in[-R,R]$, comparison principle gives
 $u(t,x)\ge q(t)$. Since $q(t)\to1-\eta$ as $t\to + \infty$,  we have 
 $$
 \liminf_{t\to\infty}\inf_{|x|\le R}u(t,x)\ge1-\eta.
 $$
 By the arbitrariness of $\eta$ and $u\le 1$, we prove \eqref{eq:zero-mean-local}.

We give the upper estimate of linear speed. 
For  $\lambda>0$, define
$$
B(\lambda):=\int_{-1}^{1}(e^{\lambda y}-1)\dd y
=\frac{e^{\lambda }- e^{ -\lambda } }{\lambda}-2\ge0.
$$
By direct computation, we note that 
$$  W(t,x) = \exp\{  -\lambda  x + t + (1-e^{-t}) B(\lambda)  \} $$
satisfies 
$$ W_t  -  e^{-t}\int_{-1}^{1}[W(t,x-y)-W(t,x)]\dd y- W(t,x)  =0,  \quad \forall t>0, x\in \R. $$
Then  the function
$$
\overline u_\lambda(t,x):=
\min\{1,  A W(t,x)]\}
$$
is a supersolution  of \eqref{eq:zero-mean-model}. Let $A>0$ be sufficiently large such that $u_0(x)\ge \overline{u}_\lambda(0, x)$ for all $x\in\R$.  Applying comparison principle  yields 
\begin{equation}\label{eq:zero-mean-exponential-upper}
u(t,x)\le\min\{1,  A \exp\left(   -\lambda  x +t +   (1-e^{-t}) B(\lambda)\right) \},  
\qquad  \forall  \lambda>0.
\end{equation}
For a given $c>0$,  choose $\lambda>1/c$.  For $x\ge ct$, we have  
$$  u(t,x) \leq A \exp \{  (1-\lambda c ) t + B(\lambda)  \} \to 0, \quad \text{ as } t\to +\infty.$$
By symmetric argument, we prove  \eqref{eq:zero-mean-linear-upper}.  The proof is complete. 
 \end{proof}

 Next we use an example to illustrate the effect of transient long distance dispersal.  
 \begin{example} \label{ex1}
 Let $\delta=1/100$ and $f(t,u)=1-u$.
 For $\alpha=0$ or $\alpha=5$, set
\begin{align*}
 B(y)=&\delta\mathbf1_{[0,1]}(y), \qquad   Q_\alpha(t,y)
  =e^{-y}\mathbf1_{[0,\infty)}(y)
   \mathbf1_{\{0<t<e^{-\alpha y}\}},  \\
& K_\alpha(t,y)=B(y)+Q_\alpha(t,y).
\end{align*}
Both kernels satisfy Assumptions~\ref{ass:f}
 and \ref{ass:k}, with
\begin{align*}
\| K_\alpha(\cdot, y) \|_{L^\infty(0, \infty)} =B(y)+e^{-y}\mathbf1_{[0,\infty)}(y),
 \qquad
 \overline K_\alpha(y):= \av{K_{\alpha} (\cdot, y)}=B(y).
\end{align*}
Moreover,  $\mathcal{I}_{K_0} = \mathcal{I}_{K_5} = (-\infty, 1)$ and  $\mathcal{I}_{\overline{K}_0} =  \mathcal{I}_{\overline{K}_5} = \R $,
where
$$\mathcal I_{K_\alpha} :=\left\{\lambda\in\R:
\int_\R \| K_\alpha (\cdot, y) \|_{L^\infty(0, \infty)} e^{\lambda y}\dd y<\infty\right\}, $$
and 
$$\mathcal I_{\overline{K}_\alpha} :=\left\{\lambda\in\R:
\int_\R \overline{K}_\alpha (y) e^{\lambda y}\dd y<\infty\right\}.  $$

  Let $u$ and $\widetilde u$ be the solutions of \eqref{Pb} 
   corresponding to $K_0$ and $K_5$, respectively,
   with the same nontrivial initial function
   $u_0\in BUC(\mathbb R, [0,1])$, supported in $[-L,0]$.
   Set 
    $$
    \hat{c}_{+} 
    :=\inf_{\lambda>0}\frac{H_n (\lambda)}{\lambda} =\inf_{\lambda>0}\frac{H (\lambda)}{\lambda},   \quad  \text{ for all large } n.
    $$
    Then 
    \begin{equation}\label{ccc}
    \hat{c}_{ +} < \frac{1}{2} < c^*_{+}.
    \end{equation}
    Moreover, we have 
    \begin{equation}\label{exexex}
   \lim_{t\to\infty}u\left(t,\frac{t}{2}\right)=1,
     \qquad
     \lim_{t\to\infty}\widetilde u\left(t,\frac{t}{2}\right)=0.
    \end{equation}
  
\end{example}

  \begin{remark}
  The kernels $K_0$ and $K_5$ 
  coincide for all $t\ge1$ and 
  $$  \av{K_0(\cdot, y)} = \av{K_5(\cdot, y)} \quad \text{ and } \quad \|K_0(\cdot, y)\|_{L^\infty(0, \infty)}=  \|K_5(\cdot, y)\|_{L^\infty(0, \infty)}. $$ 
  Nevertheless, their solutions have  different
  long time behaviors along the same path $x=t/2$.
  Hence,  the duration of transient
  long distance dispersal can affect  propagation at large time.
 
 This example also shows the importance of Assumption~\ref{ass:endpoint} in the lower estimate of spreading speed. 
  \end{remark}

In this example,  we extend $H(\lambda) $ to all $\lambda\in \R$ since $\mathcal{I}_{\overline{K}_0} =  \mathcal{I}_{\overline{K}_5} = \R$.   We note 
$$  \int_{\R}\overline{K}_\alpha(y) [e^{\lambda y} -1] \dd y + 1 = \delta\left(\frac{e^\lambda-1}{\lambda}-1\right) +1=: H(\lambda),  \quad \text{ for  } \alpha=0 \text{ or  } \alpha= 5,  $$
with $H(0)=1$.
Both kernels have the same candidate speed
$$c^*_+ =\inf_{0<\lambda<1} \frac{H(\lambda)}{\lambda} = H(1) = \delta(e-2) +1 > \delta= H'(1).  $$
 In the positive direction,  both kernels are in case $(iv)$ of Proposition~\ref{prop:classification}. In the negative direction, 
 both kernels belong to Case~$(ii)$,  with  $c^*_{-}=0$.  
Neither kernel satisfies   \eqref{eq:tail-condition} 
 in Assumption~\ref{ass:endpoint}.  Below we prove \eqref{ccc} and \eqref{exexex}.

 \begin{proof}
  Note that $H''(\lambda)>0$. Define 
 $$
 g(\lambda)=\lambda H'(\lambda)-H(\lambda)= \frac{e^{\lambda} (\lambda-2)\delta + (2+\lambda) \delta - \lambda}{\lambda}.
 $$
 Then
 $$
 g'(\lambda)=\lambda H''(\lambda)>0
 \qquad \forall \lambda>0.
 $$
 Since
 $$
 g(5)=\frac{3e^5- 493}{500}<0,
 \qquad
 g(6)=\frac{e^6-148}{150}>0,
 $$
 there exists a  unique $\hat{\lambda} \in (5, 6)$ such that 
 \begin{equation*}
  \frac{H(\hat{\lambda})}{\hat{\lambda} }= \inf_{\lambda>0} \frac{H(\lambda)}{\lambda}=\hat{c}_{  +}.
 \end{equation*}
 Hence, 
 $$ \hat{c}_{ +} < \frac{H(5)}{5} = \frac{e^5+494}{2500} < \frac{1}{2}<c_+^*. $$
 This yields that the truncated speeds  do not converge to $c^*_+$.

  Now we consider $\widetilde u$.  Let 
  $$A(t):= \int_{0}^{t} \int_{\R}K_5(s,y) e^{5y} \dd y \dd s - \int_{0}^{t} \int_{\R}K_5(s,y) \dd y \dd s+ t. $$
For almost every $t>0$, the function  $v(t,x)= e^{-5x+A(t)} $ solves 
     \begin{equation*}
     v_t= \int_{\R} K_5(t,x-y) [v(t,y) - v(t,x)] \dd y + v. 
     \end{equation*}
  Since 
  $$\int_{0}^{1} \int_{\R}K_5(s,y) e^{5y} \dd y \dd s = \delta \frac{e^5-1}{5}+ \int_{0}^{\infty}\int_{0}^{e^{-5y}} e^{4y} \dd s \dd y = \delta \frac{e^5-1}{5} +1 < \infty,  $$
  the function $\overline{v}:=\min \{ 1,  M v\}$  with large $M>0$  is a supersolution.  
  The comparison principle implies that
  $$  
  \tilde{u}(1,x) \leq  \overline{v}(1,x) \leq  Ce^{-5x} 
  $$
  for some $C>0$.

  For $t\geq 1$, the kernel  $K_5(t,y)= B(y)$.   The exponential supersolution implies 
  \begin{equation*}
  \widetilde{u}(t,x) \leq C\exp\left\{ -5x+H(5) (t-1) \right\}  \quad \forall t\geq 1, x\in\R.
  \end{equation*}
Since 
$$
  \frac{H(5)}{5} < \frac{1}{2},
 $$
 we obtain
 $$
 \widetilde u\left(t,\frac{t}{2} \right)
    \le C \exp\left\{ -5 \left(  \frac{t}{2} - \frac{H(5)}{5} (t-1) \right)  \right\} \to 0,  \qquad \text{ as }  t\to + \infty.
 $$

 Next  we  consider $u$.
   For $0<t<1$, we have 
   $$
   K_0(t,y)=B(y)+Q(y),
   \qquad
   Q(y)=e^{-y}\mathbf1_{[0,\infty)}(y).
   $$
Let $\underline{v}$ solve
$$\underline{v}_t (t,x)= \int_{\R}Q(y)\underline{v}(t,x-y)\dd y - \int_{\R}[B(y)+Q(y)]\dd y  \; \underline{v}(t,x).$$
Since $u(1-u)\ge0$ on $[0,1]$, comparison principle yields that 
$$ u(1,x) \ge e^{-(1+\delta)}  \int_{\R}Q(x -y) u_0(y)\dd y, \quad \forall x\in \R.$$
  Hence,   for every $x\ge0$,
   $$
   u(1,x) \ge e^{-(1+\delta)}
   \int_{-L}^0e^{-(x-z)}u_0(z)\dd z=C_0e^{-x},
   $$
   where
   $$
   C_0=e^{-(1+\delta)}
   \int_{-L}^0e^zu_0(z)\,dz>0.
   $$
   
   For $t\ge1$, the kernel is $B$, and hence
   $$
   u_t=B*u-\delta u+u(1-u)
   \ge u(1-\delta-u).
   $$
 Let  $\rho:=1-\delta$.
  Consider the solution of logistic equation $z_t= z(\rho -z )$, starting from    $z(1,x)=C_0e^{-x}$ for $x\ge0$.  
  By the  comparison principle, we have 
  $$
   u(t,x)\ge
   \frac{\rho}{
   1+\left(\frac{\rho  -C_0e^{-x}}{C_0 e^{-x}}\right)
   e^{-\rho(t-1)}  }  =  \frac{\rho}{
      1+\left(\frac{\rho e^{x} }{C_0}  -1 \right)
      e^{-\rho(t-1)}  },
   \qquad t\ge1,\ x\ge0.
   $$
   For any $0<a_0<b_0<\rho$, one has 
   $$ \lim_{t\to \infty}  \inf_{a_0 t \le  x\leq b_0t}\frac{\rho}{
         1+\left(\frac{\rho e^{x} }{C_0}  -1 \right)
         e^{-\rho(t-1)}  }= \rho. $$
Hence, we know 
   $$
   \liminf_{t\to\infty}
   \inf_{a_0t\le x\le b_0t}u(t,x)
   \ge\rho = 1-\delta >0.
   $$
   Since $1-\delta >1/2$, we choose $a_0, b_0$ such that $0 < a_0 <1/2 <b_0 <1-\delta$. 
  Lemma~\ref{lem:upgrade}  yields
   $$
   \lim_{t\to\infty}u\left(t,\frac{t}{2} \right)=1.
   $$
   The proof is complete.
 
 \end{proof}

The last  example  shows  how 
the minimizing parameters of variational formula  depend on the time changing. 
 \begin{example} \label{ex:periodic-direction-switching}
 Let
 $$
 J_+(y)=\frac{e^{-y}}{(1+y)^3}\mathbf{1}_{(0,\infty)}(y)
 $$
 and, for $0<\vartheta<1$, define the $1$-periodic function
\begin{equation*}
 a_\vartheta(t)=
 \begin{cases}
 1, & t\!\!\pmod{1}\in[0,\vartheta),\\
 0, & t\!\!\pmod{1}\in[\vartheta,1).
 \end{cases}
\end{equation*}

 Consider
 $$
 K_\vartheta(t,y)=a_\vartheta(t)J_+(y)
   +(1-a_\vartheta(t))J_+(-y),
 \qquad f(t,u)=r(1-u), \quad r>0.
 $$
 Within  each period, the dispersal is  rightward in $[0, \vartheta)$ and   leftward during  $[\vartheta, 1)$.  
 Note that 
 $$\overline{K}_{\vartheta} (y) =   \vartheta J_+(y)+(1-\vartheta)J_+(-y), 
  $$
 $$
  \| K_{\vartheta}(\cdot, y) \|_{L^\infty(0, \infty)} =J_+(y)+J_+(-y),
  \qquad  \mathcal I=[-1,1].
 $$
 \end{example}

To explain this example, we first  introduce the following notation, 
\begin{equation*}
\begin{aligned}
 H_\vartheta(\lambda)=
   \int_{\mathbb R}\overline K_\vartheta(y)
   (e^{\lambda y}-1)\dd y+r = \int_{0}^{\infty}   \frac{\vartheta e^{(\lambda-1)y} +(1-\vartheta)e^{-(\lambda+1)y}  -e^{-y} }{(1+y)^3} \dd y +r,\\
     g_\vartheta(\lambda)=
      \lambda H_\vartheta'(\lambda)-H_\vartheta(\lambda)=   \int_{0}^{\infty}   \frac{ \vartheta  ( \lambda y -1) e^{(\lambda-1)y} -(1-\vartheta) (\lambda y +1) e^{-(\lambda+1)y}  +e^{-y} }{(1+y)^3} \dd y -r.
\end{aligned}
\end{equation*}
  Define
\begin{equation*}
 \begin{aligned}
 A&:=\int_0^\infty
   \frac{y-1+e^{-y}}{(1+y)^3}\dd y,\\
 B&:=\int_0^\infty
   \frac{e^{-y}-(1+y)e^{-2y}}{(1+y)^3}\dd y.
 \end{aligned}
\end{equation*}
 Both constants are positive, and $A>B$ since 
 $$
 A-B=\int_0^\infty
  \frac{y-1+(1+y)e^{-2y}}{(1+y)^3}\dd y>0.
$$
  
 A direct computation gives
\begin{equation*}
 \begin{aligned}
 g_\vartheta(1)&=\vartheta A+(1-\vartheta)B-r,\\
 g_\vartheta(-1)&=\vartheta B+(1-\vartheta)A-r.
 \end{aligned}
\end{equation*}
 In particular, choosing $r=(A+B)/2$ yields
$$
 g_\vartheta(1)=(\vartheta-\tfrac12)(A-B),
 \qquad
 g_\vartheta(-1)=(\tfrac12-\vartheta)(A-B).
 $$
 Since $g_\vartheta'(\lambda)=\lambda H_\vartheta''(\lambda)$
 and $g_\vartheta(0)=-r<0$, Proposition~\ref{prop:classification}
 implies the following classification:
 \begin{itemize}
 \item If $0<\vartheta<1/2$, the rightward infimum of $H(\lambda)/ \lambda$ is approached as $\lambda \to 1$, with  
 $H_\vartheta'(1)<H_\vartheta(1)$  (Case $(iv)$);  the leftward infimum is attained  in $(-1, 0)$ (Case $(i)$).

 \item If $\vartheta=1/2$, both infima are approached at finite endpoints (Case  $(iii)$).
 
 \item If $1/2<\vartheta<1$, the rightward infimum is attained in $(0, 1)$ (Case $(i)$); the leftward infimum is approached as $\lambda \to -1$ (Case $(iv)$). 
 
 \end{itemize}
Since $\mathcal{I}_{\rm av} = \mathcal{I}=[-1,1]$, Assumption \ref{ass:endpoint} is satisfied.  We can apply Theorem~\ref{thm:spreading} to obtain  spreading speeds.

\section*{Declaration of AI-Assisted Use}
The authors used  ChatGPT  to improve the language and  mathematical exposition, and to  refine the presentation of theorems and proofs. 
In particular, AI assisted in formulating  Assumption~\ref{ass:endpoint} and  suggested   Examples~\ref{ex1} and \ref{ex:periodic-direction-switching}. 
The  main mathematical ideas and theoretical results were developed by the authors. 
The authors have independently written all mathematical statements, proofs, computations, and references, and
take full responsibility for the content of this manuscript.

\section*{Acknowledgements}
 Z. Jin  is partially supported by the Strategic Priority Research  Program of CAS (No. XDB0900100), and NSFC (Nos. 12401225 and  12471150).   T. Zhou  is supported by NSFC (No. 12471150).

\begin{small}

\end{small}

\end{document}